\documentclass[10pt]{amsart}
\usepackage{amsmath, amssymb}
\usepackage{amsfonts}
\usepackage{mathrsfs}
\usepackage[color, arrow,matrix,curve,cmtip,ps]{xy}
\usepackage{paralist}
\usepackage{amsthm}
\usepackage{tikz-cd}
\usetikzlibrary{arrows,calc,matrix}
\tikzset{
curvarr/.style={
  to path={ -- ([xshift=2ex]\tikztostart.east)
    |- (#1) [near end]\tikztonodes
    -| ([xshift=-2ex]\tikztotarget.west)
    -- (\tikztotarget)}
  }
}

\usepackage[final]{pdfpages}
\usepackage{dsfont,txfonts}
\usepackage{tikz}
\usepackage{rotating}
\usepackage{mathrsfs}
\PassOptionsToPackage{hyphens}{url}\usepackage{hyperref}
\usepackage{enumitem}
\usepackage{graphicx}
\usepackage{mathtools}
\usetikzlibrary{arrows,chains,matrix,positioning,scopes}
\usepackage{listings}
\lstdefinelanguage{GAP}{%
  morekeywords={%
    Assert,Info,IsBound,QUIT,%
    TryNextMethod,Unbind,and,break,%
    continue,do,elif,%
    else,end,false,fi,for,%
    function,if,in,local,%
    mod,not,od,or,%
    quit,rec,repeat,return,%
    then,true,until,while%
  },%
  sensitive,%
  morecomment=[l]\#,%
  morestring=[b]",%
  morestring=[b]',%
}[keywords,comments,strings]

\usepackage[T1]{fontenc}
\usepackage[variablett]{lmodern}
\usepackage{xcolor}
\usepackage[citestyle=alphabetic,bibstyle=alphabetic]{biblatex}
\usepackage{lipsum}

\newtheorem{theorem}{Theorem}[section]
\theoremstyle{definition}
\newtheorem{lemma}[theorem]{Lemma}

\newtheorem{conjecture}[theorem]{Conjecture}

\newtheorem{proposition}[theorem]{Proposition}
\newtheorem{corollary}[theorem]{Corollary}
\newtheorem{definition}[theorem]{Definition}

\newtheorem{remark}[theorem]{Remark}

\newtheorem{example}[theorem]{Example}

\everymath{\displaystyle}
\allowdisplaybreaks

\newcommand{\set}[1]{\left\lbrace #1 \right\rbrace}
\newcommand*{\defeq}{\mathrel{\vcenter{\baselineskip0.5ex \lineskiplimit0pt
                     \hbox{\scriptsize.}\hbox{\scriptsize.}}}%
                     =}

\begin{document}
\title[The passage from the integral to the rational group ring]{The passage from the integral to the rational group ring in algebraic $K$-theory}
\author{Georg Lehner}
\subjclass[2020]{Primary 19A31; Secondary 20C07, 19D35, 55P42.}
\keywords{$K$-theory of group rings, Farrell-Jones conjecture, negative $K$-theory.}
\begin{abstract}
An open question is whether the map $\widetilde{K_0 }\mathbb{Z} G \rightarrow \widetilde{K_0 }\mathbb{Q} G$ in reduced $K$-theory from the integral to the rational group ring is trivial for any group $G$. We will show that this is false, with a counterexample given by the group $QD_{32} *_{Q_{16}} QD_{32}$. We will also show how to compute the image of the map $\widetilde{K_0 }\mathbb{Z} G \rightarrow \widetilde{K_0 }\mathbb{Q} G$ using representation theoretic means, assuming $G$ satisfies the Farrell-Jones conjecture.
\end{abstract}
\maketitle

\section{Introduction}

Let $G$ be a group, $R$ a ring and define $RG$ to be the group algebra of $G$ over $R$. The algebraic $K$-theory group is defined as
$$K_0 R G \defeq \mathbb{Z}\set{\text{isomorphism classes of f.g.\ projective}~R G\text{-modules} }/\equiv$$
where $\equiv$ is the equivalence relation generated by $[ A \oplus B ] \equiv [A] + [B]$. 

For $R$ being the ring of integers, the subgroup of $K_0 \mathbb{Z} G$ generated by the free modules is always a split summand isomorphic to $\mathbb{Z}$. The \emph{reduced} $K$-theory group $\widetilde{ K_0 }\mathbb{Z} G$ is defined as the quotient of $K_0 \mathbb{Z} G$ by this summand.

The group $\widetilde{ K_0 } \mathbb{Z} G$ is an important invariant of $G$ appearing in a variety of geometric problems, most notably as the group containing Wall's finiteness obstruction. Wall \cite{wall1965finiteness} showed that for every finitely dominated CW-complex $X$ with fundamental group $G = \pi_1(X)$, there exists an element $w(X)$ in $\widetilde{ K_0 } \mathbb{Z} G$ which is trivial iff $X$ is actually finite. Moreover, every element of $\widetilde{ K_0 } \mathbb{Z} G$ can be realized in this way from some finitely dominated CW-complex.\footnote{See also \cite{varadarajan1989finiteness}}

However, $\widetilde{ K_0 } \mathbb{Z} G$ tends to be very hard to compute in general. One of the few structural things that can be said about $\widetilde{ K_0 }\mathbb{Z} G$ is a theorem due to Swan that for $G$ being a finite group $\widetilde{ K_0 } \mathbb{Z} G$ is finite.

Changing the base ring to the rational numbers, we can define $\widetilde{ K_0 }\mathbb{Q} G$ in a similar manner. As before, we obtain a splitting $K_0 \mathbb{Q} G \cong \mathbb{Z} \oplus \widetilde{ K_0 }\mathbb{Q} G$. For $G$ a finite group, $K_0 \mathbb{Q} G$ is inherently easier to compute than its integral counterpart. Since in this case $\mathbb{Q} G$ is a finite dimensional semisimple algebra over $\mathbb{Q}$, it splits as a product of matrix algebras over division algebras over $\mathbb{Q}$, one for each irreducible $\mathbb{Q}$-representation of $G$. This means that $ K_0 \mathbb{Q} G \cong \mathbb{Z}^{r_\mathbb{Q}}, $ where $r_\mathbb{Q}$ is the number of isomorphism classes of irreducible $\mathbb{Q}$-representations of $G$. In particular, $\widetilde{ K_0 }\mathbb{Q} G$ is a free abelian group of rank $r_\mathbb{Q} - 1$.

We can thus see that for $G$ being a finite group, the map $K_0 \mathbb{Z} G \rightarrow K_0 \mathbb{Q} G$ defined via $ [P] \mapsto [P \otimes \mathbb{Q}] $, for $P$ being a f.g.\ projective $\mathbb{Z}G$-module, is an isomorphism on the summands corresponding to free modules over $\mathbb{Z}G$ and $\mathbb{Q}G$ respectively, and trivial on the quotients
$$ \widetilde{ K_0 }\mathbb{Z} G \rightarrow \widetilde{ K_0 }\mathbb{Q} G,$$
since there it is a homomorphism from a finite to a free abelian group. Swan also showed the following slightly stronger result.

\begin{theorem}[Swan, \cite{Swan1960InducedRA}] \label{swansprojectivetheorem}
Suppose $G$ is a finite group and $P$ a finitely generated projective $\mathbb{Z} G$-module. Then $P \otimes \mathbb{Q}$ is free.
\end{theorem}

The statement that $\widetilde{ K_0 }\mathbb{Q} G$ is free does not generalize to arbitrary groups. In fact, Kropholler, Moselle \cite{kropholler_moselle_1991}, and Leary \cite{learytorsionprojective} constructed specific examples of groups which have $2$-torsion elements in $K_0 \mathbb{Q} G$. This means we cannot expect to find a straightforward generalization of Swan's theorem to infinite groups.

Bass \cite{Bass1976EulerCA} investigated this question and formulated what is now known as the strong Bass conjecture for $K_0 \mathbb{Z} G$. For this, let $r  \colon  K_0 \mathbb{Z} G \rightarrow HH_0 (\mathbb{Z}G)$ denote the Hattori-Stallings trace map and define $r_P(g)$ as the coefficient in the sum
$$ r(P) = \sum_{[g] \in \text{Conj}(G)} r_P(g) $$
under the isomorphism $HH_0(RG) = \bigoplus_{ \text{Conj}(G) } R$. This gives a function $r_P : G \rightarrow \mathbb{Z}$.

\begin{conjecture}[Strong Bass Conjecture for $K_0 \mathbb{Z} G$, \cite{Bass1976EulerCA}]
The value $r_P(g)$ is $0$ for $g \neq 1$.
\end{conjecture}

Lück, Reich \cite{Lueck_Reich_2005}, Section 3.1.3, show that the strong Bass conjecture for $K_0 \mathbb{Z} G$ follows from the stronger claim that the map
$$ \widetilde{ K_0 }\mathbb{Z} G \rightarrow \widetilde{ K_0 }\mathbb{Q} G $$
vanishes rationally, and they also show that this holds true if $G$ satisfies the Farrell-Jones conjecture.
\begin{theorem}[\cite{Lueck_Reich_2005}, Proposition 3.11] \label{FJimpliesrationalvanishing}
Assume $G$ satisfies the Farrell-Jones conjecture. Then the map $\widetilde{ K_0 }\mathbb{Z} G \otimes \mathbb{Q} \rightarrow \widetilde{ K_0 }\mathbb{Q} G \otimes \mathbb{Q}$ is trivial.
\end{theorem}
We will give a definition of the Farrell-Jones conjecture below. In Remark 3.13 they ask whether this is true integrally. This conjecture was also published in \cite{Reich2008}.

\begin{conjecture}[Integral $\widetilde{ K_0 }\mathbb{Z} G$-to-$\widetilde{ K_0 }\mathbb{Q} G$-conjecture]
The map $\widetilde{ K_0 }\mathbb{Z} G \rightarrow \widetilde{ K_0 }\mathbb{Q} G$ is trivial.
\end{conjecture}

Part of this paper will show the following.

\begin{theorem}[See Section \ref{counterexample}]
The Integral $\widetilde{ K_0 }\mathbb{Z} G$-to-$\widetilde{ K_0 }\mathbb{Q} G$-Conjecture is false. A counterexample is given by the group $G \defeq QD_{32} \ast_{Q_{16}} QD_{32}$, where $QD_{32}$ is the quasi-dihedral group of order $32$, and $Q_{16}$ is the generalized quaternion group of order $16$.
\end{theorem}

The other half of this paper is an investigation into how much the map $\widetilde{ K_0 }\mathbb{Z} G \rightarrow \widetilde{ K_0 }\mathbb{Q} G$ can fail to be trivial under the assumption that $G$ satisfies the Farrell-Jones conjecture. We recommend the surveys \cite{Lueck_Reich_2005} as well as \cite{reich_varisko} for an overview on assembly conjectures, and in particular the Farrell-Jones conjecture. The Farrell-Jones conjecture states that the assembly map
$$ \text{colim}_{(G/H) \in \text{Or}G_{\text{VCyc}}} \textbf{K} \mathbb{Z}H \rightarrow \textbf{K} \mathbb{Z}G $$
is a weak equivalence of spectra. Here $\textbf{K} \mathbb{Z}G $ refers to the non-connective $K$-theory spectrum of $\mathbb{Z}G$, and the colimit in question is a homotopy colimit over the category $\text{Or}G_{\text{VCyc}}$, which is the full subcategory of the orbit category of $G$ spanned by the objects $G/H$ with virtually cyclic isotropy group $H$. The Farrell-Jones conjecture has been shown to be true for a wide class of groups by the work of Bartels, Lück, Reich \cite{Bartels_2007}, Bartels, Bestvina \cite{Bartels2019} Kammeyer, Lück,  Rüping \cite{Kammeyer_2016}, and Wegner \cite{Wegner_2015} among many others.

Now let $E \text{Fin}$ be a fixed model for the classifying space of finite subgroups together with a chosen CW-structure $(E \text{Fin}^{(k)})_{k \in \mathbb{N}}$. Write
$$ (f,g)  \colon \coprod_{i \in I} G/H_i \times S^0 \rightarrow \coprod_{j \in J} G/K_j $$
for the degree $0$ attaching map of $E\text{Fin}$ with the $H_i$ and $K_j$ being finite subgroups of $G$.
For a functor $F  \colon \text{Or}G \rightarrow \text{Ab}$ define
$$ \text{ker}^F \defeq \text{ker}( F(f)-F(g) )  \colon \bigoplus_{i \in I} F(G/H_i) \rightarrow \bigoplus_{j \in J} F(G/K_j). $$

\begin{theorem} 
Suppose $G$ satisfies the Farrell-Jones conjecture. There is an exact sequence
$$ 0 \rightarrow \normalfont{\text{ker}}^{\widetilde{K_0 }\mathbb Q}  \rightarrow \normalfont{\text{ker}}^{\text{SC}} \rightarrow \normalfont{\text{ker}}^{K_{-1} \mathbb{Z}} \rightarrow \normalfont{\text{im}}( \widetilde{K_0 }\mathbb{Z}G \rightarrow \widetilde{K_0 }\mathbb{Q}G ) \rightarrow 0. $$
\end{theorem}

This gives a certain limitation on the map $ \widetilde{ K_0 }\mathbb{Z} G  \rightarrow \widetilde{ K_0 }\mathbb{Q} G$. The terms $\text{ker}^{\widetilde{K_0 }\mathbb Q}, \text{ker}^{\text{SC}}$ and $\text{ker}^{K_{-1} \mathbb{Z}}$ are computable by representation theoretic techniques and this is what allowed the computation of the above counterexample. The groups $\text{ker}^{\widetilde{K_0 }\mathbb Q}$ and $\text{ker}^{\text{SC}}$ are always free, and $\text{ker}^{K_{-1} \mathbb{Z}}$ is free $p$-locally away from the prime $2$. We will define the functor $\text{SC}$ in Section \ref{whiteheadspectrum} and give a characterization in terms of $p$-adic characters for finite groups in Section \ref{singularcharactergroup}.

One could still ask if the image of the map $\widetilde{K_0 }\mathbb{Z}G \rightarrow \widetilde{K_0 }\mathbb{Q}G$ is in fact only $2$-torsion. The author at present believes this to be false. Examples of groups with odd torsion are however more challenging to construct and will be saved for a later publication.

\subsection{Acknowledgements}

This paper is part of the results of my thesis. I want to thank my supervisor, Holger Reich, as well as my colleagues Elmar Vogt, Vincent Boelens, Gabriel Angelini-Knoll and Alexander Müller. I want to thank Maxime Ramzi for pointing out the connection to the Bass conjecture to me, as well as Johnny Nicholson for his input on groups failing the Eichler condition. This work was funded by the German Research Foundation through the Graduate School ``Berlin Mathematical School''.

\section{Preliminaries}

Throughout, we will denote the \emph{non-connective algebraic} $K$-\emph{theory spectrum} of a non-commutative unital ring $R$ by $\textbf{K}R$, see e.g. \cite{Weibel2013TheKA}, Chapter IV. Its homotopy groups $K_n R \defeq \pi_n \textbf{K} R$ are the algebraic $K$-theory groups of $R$.

Sections \ref{ghomology} to \ref{whiteheadspectrum} are used to phrase and setup the Farrell-Jones conjecture and discuss how to deal with functors on the orbit category $\text{Or}G$. Section \ref{lowerktheory} is concerned with primarily classical results about lower $K$-theory groups of finite groups $G$. The proofs of the main theorems will be found in section \ref{infinitegroups}. The claimed counterexample to the integral $K_0 \mathbb{Z}G$-to-$K_0 \mathbb{Q}G$ conjecture is discussed in Section \ref{counterexample}.

We will use the language of $\infty$-categories in our proofs. The author remarks that this choice is due to convenience, not necessity. The reader not familiar with the topic shall be assured that all arguments can be phrased using the notions of model categories and $t$-structures on a triangulated category, only that many formal arguments become harder to formulate (e.g. the existence of the object-wise $t$-structure on a functor category or exactness of many functors involved).  A model for the notion of $\infty$-categories is given by the notion of quasi-categories developed by Joyal and Lurie, which are simplicial sets fulfilling a certain lifting property. The standard reference is \cite{luriehtt}. We further included some results used about $t$-structures on stable $\infty$-categories in Appendix \ref{tstructures}. We want to remark that most of our results will be phrased in a model independent way, treating the notion of $\infty$-categories as a black box. The terms \emph{limit} and \emph{colimit} will always be interpreted in an $\infty$-categorical way. In situations where our $\infty$-category $\mathcal{C}$ arises from a model category $\mathcal{M}$, limits and colimits in $\mathcal{C}$ are modelled by homotopy limits and homotopy colimits in $\mathcal{M}$. If $C$ is a $1$-category, then the nerve $N(C)$ is an $\infty$-category in which limits model ordinary $1$-categorical limits and similarly for colimits. We will often omit the notation for the nerve and just refer to the $\infty$-category $N(C)$ simply as $C$ when the context is clear. Given two $\infty$-categories $\mathcal{C}$ and $\mathcal{D}$, there is the $\infty$-category of functors $\text{Fun}(\mathcal{C}, \mathcal{D})$ from $\mathcal{C}$ to $\mathcal{D}$. Functors $A  \colon  \mathcal{C} \rightarrow \mathcal{D}$ are sometimes written as $A(-)$ to highlight that the value of $A$ is dependent on the input. Natural transformations between functors will be depicted with a double arrow, like $A \implies B$.

The $\infty$-category of \emph{spaces}, sometimes also referred to as \emph{homotopy types} or \emph{anima}, will be denoted as Spc and is characterized via a universal property as the free cocomplete $\infty$-category generated by a single object (the point) similar to how the category of sets is generated under coproducts by a single object (the set with a single element). It is modelled, for example, by the simplicial category of Kan complexes or the relative category of CW-complexes and weak equivalences being homotopy equivalences. The undercategory $\text{Spc}_{\text{pt}/}$ is called the $\infty$-category of \emph{pointed spaces} and will be denoted as $\text{Spc}_{*}$. We have a natural functor $(-)_+  \colon  \text{Spc} \rightarrow \text{Spc}_*$ that adds a disjoint basepoint. The $\infty$-category of spectra will be denoted as Sp and is characterized again via a universal property as the stabilization of Spc, or equivalently, as the free cocomplete stable $\infty$-category generated by a single object (the sphere spectrum $\mathbb{S}$). It is modelled, for example, by the relative category of $\Omega$-spectra and weak equivalences being maps that induce isomorphisms on all homotopy groups. Spectra will be denoted by bold-faced letters, e.g. $\textbf{A}, \textbf{K}R$ or $\textbf{Wh}(R;G)$. Since Sp is stable, the suspension $\textbf{A} \mapsto \Sigma \textbf{A}$ defined as the pushout
$$\xymatrix{
\textbf{A} \ar[r] \ar[d] & 0 \ar[d] \\
0 \ar[r] & \Sigma \textbf{A}
}$$
produces an auto-equivalence of Sp with itself. The suspension functor $\text{Spc}_* \rightarrow \text{Sp}$ will be denoted as $\Sigma^\infty$, the composition $\Sigma^\infty\circ (-)_+$ will be denoted as $\Sigma^\infty_+$. The homotopy groups of a spectrum $\textbf{A}$ are denoted as  $\pi_n(\textbf{A})$. Further, denote the $1$-category of abelian groups by Ab. Taking homotopy groups yields functors $\pi_n  \colon  \text{Sp} \rightarrow \text{Ab}$.  Spectra $\textbf{A}$ with the property that $\pi_n \textbf{A} = 0$ for $n < 0$ will be called connective, and spectra $\textbf{A}$ with the property that $\pi_n \textbf{A} = 0$ for $n > 0$ will be called coconnective. The functor $\pi_0$ becomes an equivalence when restricted to the intersection of the full subcategories of connective and coconnective spectra (essentially a consequence of the Brown representability theorem). Its inverse will be denoted by $\textbf{H}$, or the \emph{Eilenberg-Maclane functor}. The inclusion of the full subcategory of connective spectra into Sp admits a right adjoint which will be called $\tau_{\geq 0}$, and we define for any $a \in \mathbb{Z}$ the functor $ \tau_{\geq a}$ as $\Sigma^a \tau_{\geq 0} \Sigma^{-a}$. Similarly, the inclusion of the full subcategory of coconnective spectra into Sp admits a left adjoint which will be called $\tau_{\leq 0}$, and $ \tau_{\leq a}$ is defined as $\Sigma^a \tau_{\leq 0} \Sigma^{-a}$ in the same way. For $a,b \in \mathbb{Z}$ the compositions $\tau_{\geq a} \tau_{\leq b}$ and $\tau_{\leq b} \tau_{\geq a}$ are naturally isomorphic and will denoted as $\textbf{A} \mapsto \textbf{A}[a,b]$. This type of structure defines a $t$-structure on Sp, more on this in Appendix \ref{tstructures}.

For a fixed $\infty$-category $\mathcal{C}$, for two given objects $c,c' \in \mathcal{C}$ the mapping space from $c$ to $c'$ will be denoted $\text{Map}_\mathcal{C}(c,c')$. The subscript is omitted in the case of $\mathcal{C}$ being the $\infty$-category of spaces. $ \text{Map}_\mathcal{C}$ is a bi-functor into the category Spc, contravariant in the left and covariant in the right variable. We also use the notation $[ c, c' ] \defeq \pi_0 \text{Map}_\mathcal{C}(c,c').$ Note that $[ c, c' ]$ is just the Hom-set of the $1$-category given by the homotopy category of $\mathcal{C}$. If $\mathcal{C}$ is moreover a stable $\infty$-category, the space $\text{Map}_\mathcal{C}(c,c')$ is naturally the zero-th space of a spectrum $\textbf{map}_\mathcal{C}(c,c')$, which is called the mapping spectrum from $c$ to $c'$. Again $\textbf{map}_\mathcal{C}$ is naturally a functor in both variables. We, similarly, omit the subscript in the case of $\mathcal{C}$ being the $\infty$-category of spectra. Since 
$$ [c, c'] = \pi_0 \text{Map}_\mathcal{C}(c,c') \cong \pi_0 \textbf{map}_\mathcal{C}(c,c') $$
is the zero-th homotopy group of a spectrum, the set $[c, c']$ comes naturally with the structure of an abelian group. The mapping space for two functors $F,G  \colon  \mathcal{D} \rightarrow \mathcal{C}$ in the functor category will also be denoted as
$ \text{Nat}_\mathcal{C} ( F , G )$ and is called the \emph{space of natural transformations} from $F$ to $G$. If $\mathcal{C}$ is stable, the corresponding mapping spectrum is also written as $ \textbf{nat}_\mathcal{C} ( F , G )$.

If $G$ is a group, the category $\mathcal{M}$ of topological spaces with left $G$-action admits the structure of a closed simplicial model category \cite{DWYER1984147}. The associated $\infty$-category is the $\infty$-category of $G$-spaces, $G\text{-Spc}$. Similarly, the category of pointed $G$-spaces, $G\text{-Spc}_*$ is defined as the over category $G\text{-Spc}_{\text{pt}/}$ where $\text{pt}$ is the one-point space with trivial $G$-action. Let $H$ be a subgroup of $G$. We can think of the left $G$-set $G/H$ as a discrete $G$-space. We can also interpret the $n$-sphere $S^n$ as well as the $n$-disc $D^n$ as $G$-spaces by equipping them with the trivial $G$-action. If $X$ is an object of $\mathcal{M}$, i.e. a topological space with continuous left $G$-action, a $G$-CW-structure on $X$ refers to a sequence of $G$-spaces $(X^{(k)})_{k \geq 0}$, maps $\iota_k : X^{(k)} \rightarrow X^{(k+1)}$ such that there exist pushout squares in the category $\mathcal{M}$,
$$\begin{tikzcd}
\coprod_{i \in I_k} G/H_i \times S^n \arrow[r, "{\phi_k}"] \ar[d] &  X^{(k)} \arrow[d, "{\iota_k}"] \\
\coprod_{i \in I_k} G/H_i \times D^{n+1} \ar[r] &  X^{(k+1)},
\end{tikzcd}$$
and an equivariant homeomorphism $X \cong \text{colim}_{k} X^{(k)}$, where the colimit in question is over the tower given by the maps $\iota_k  \colon  X^{(k)} \rightarrow X^{(k+1)}$. The maps 
$$\phi_k  \colon  \coprod_{i \in I_k} G/H_i \times S^n \rightarrow  X^{(k)}$$
are called \emph{attaching maps}. The indexing sets $I_k$ can be arbitrary sets. The space $X^{(k)}$ is also called the $k$-\emph{skeleton} of $X$. We also refer to $X$ together with a fixed choice of $G$-CW-structure as a $G$-CW-complex. By \cite{DWYER1984147}, Theorem 2.2, every $G$-CW-complex is a cofibrant object in $\mathcal{M}$. Moreover, every object of the $\infty$-category $G$-Spc can be represented by a $G$-CW-complex. The maps $G/H_i \times S^n \rightarrow  G/H_i \times D^{n+1}$ are cofibrations in the model category $\mathcal{M}$. This means that for a given $G$-CW-complex $X$, the squares
$$\begin{tikzcd}
\coprod_{i \in I_k} G/H_i \times S^n \arrow[r, "{\phi_k}"] \ar[d] &  X^{(k)} \arrow[d, "{\iota_k}"] \\
\coprod_{i \in I_k} G/H_i \times D^{n+1} \ar[r] &  X^{(k+1)},
\end{tikzcd}$$
are also pushout squares in the $\infty$-category $G$-Spc. If $X$ is an object of $G$-Spc we will define a $G$-CW-structure on $X$ to be a $G$-CW-structure on any representing object of $X$ in $\mathcal{M}$. We will also refer to objects of $G$-Spc from now on as $G$-spaces.

\section{$G$-homology theories and functors on the orbit category} \label{ghomology}

Define the \emph{orbit category} $\text{Or}G$ as the full subcategory of the 1-category of $G$-sets spanned by the $G$-sets with transitive action. Equivalently, each object $S$ of $\text{Or}G$ is $G$-equivariantly isomorphic to a set of left cosets $G/H$, acted on by the left, where $H$ is the isotropy group of a chosen element $s$ of $S$. It is an elementary fact that each map in $\text{Or}G$ can be decomposed into a composition of maps given by inclusions $\iota : G/H \rightarrow G/H', kH \mapsto kH'$ for $H \subset H'$ and conjugations $c_g : G/H \rightarrow G/(g^{-1}Hg), kH \mapsto kHg = (kg)(g^{-1}Hg)$. If we have a $G$-space $X$, we can think of the assignment $G/H \mapsto X^H$ as a functor
$$ X^-  \colon  \text{Or}G^{op} \rightarrow \text{Spc}.$$
Elmendorf's Theorem states that mapping $X$ to $X^-$ produces an equivalence of $\infty$-categories $G \text{-Spc} \simeq \text{Fun}(\text{Or}G^{op}, \text{Spc})$, see e.g. \cite{DWYER1984147} Theorem 3.1. Note that it further refines to an equivalence $G \text{-Spc}_* \simeq \text{Fun}(\text{Or}G^{op}, \text{Spc}_*)$ for pointed $G$-spaces. Recall that the $\infty$-category of presheaves has the following universal property.
\begin{theorem}[\cite{luriehtt}, Theorem 5.1.5.6]
Let $S$ be a small $\infty$-category, and $\mathcal{C}$ be an $\infty$-category admitting all small colimits. Then the Yoneda embedding $y : S \rightarrow \mathrm{Fun}(S^{\mathrm{op}}, \mathrm{Spc})$ induces an equivalence of categories
$$\mathrm{Fun}^L( \mathrm{Fun}(S^{\mathrm{op}}, \mathrm{Spc}), \mathcal{C}) \simeq \mathrm{Fun}(S, \mathcal{C}),$$
where $\mathrm{Fun}^L( \mathrm{Fun}(S^{\mathrm{op}}, \mathrm{Spc}))$ is the subcategory of $\mathrm{Fun}( \mathrm{Fun}(S^{\mathrm{op}}, \mathrm{Spc}))$ spanned by the colimit preserving functors. The inverse is given by left Kan extending a functor defined on $S$ along the Yoneda embedding.
\end{theorem}

\begin{definition}[Orbit tensor product]
Let $\textbf{A}$ be a functor $\text{Or}G \rightarrow \text{Sp}$. Denote by
$$ - \otimes_{{\text{Or}G}} \textbf{A} : G \text{-Spc} \rightarrow \text{Sp}$$
the left Kan extension of $\textbf{A}$, provided by the previous theorem. We write
$$ X \otimes_{{\text{Or}G}} \textbf{A}$$
for its value on a $G$-space $X$, and call it the \emph{orbit tensor product}.
In the special case of $G = \{1\}$ being the trivial group, we simply write
$$ X \otimes \textbf{A} \in \text{Sp} $$
for $X \in \text{Spc}$ and $\textbf{A} \in \text{Sp}$ and obtain the \emph{tensoring} of spectra with spaces. We note that
$$ X \otimes \textbf{A} \simeq \Sigma^\infty_+ X \otimes \textbf{A},$$
where the tensor product on the right refers to the smash product of spectra.
\end{definition}

If $H \subset G$ a subgroup, the functor $\mathrm{Ind}_H^G : \text{Or}H \rightarrow \text{Or}G$ induces a colimit preserving functor
$$ \mathrm{Ind}_H^G : H \text{-Spc} \rightarrow G \text{-Spc},$$
by extending the composite $\text{Or}H \rightarrow \text{Or}G \xrightarrow{y} G \text{-Spc}$. In this context the projection formula holds.

\begin{proposition}[Projection formula] \label{projectionformula} Let $H \subset G$ a subgroup, $X \in H \text{-Spc}$, $\mathbf{A} \in \text{Or}G \rightarrow \text{Sp}$. Then there is a natural equivalence
$$ (\mathrm{Ind}_H^G X) \otimes_G \mathbf{A} \simeq X \otimes (A \circ \mathrm{Ind}_H^G). $$
\end{proposition}

\begin{proof}
Since $(\mathrm{Ind}_H^G -) \otimes_G \mathbf{A}$ is by construction colimit preserving, it suffices to verify the equivalence on representables. On representables the equivalence is true by definition.
\end{proof}

It is straightforward to see that the functors $ A_* \defeq \pi_* ( - \otimes_{{\text{Or}G}} \textbf{A} )$ define a $G$-equivariant homology theory on $G$-CW-complexes in the sense of Lück \cite{lueck2019assembly}, Definition 2.1. Moreover, we can equip $\text{Fun}(\text{Or}G, \text{Sp})$ with the object-wise $t$-structure defined in Section \ref{objectwisetstructure}. Since connective objects are closed under colimits, we get that if $\textbf{A}$ is object-wise connective, then $X \otimes_{{\text{Or}G}} \textbf{A}$ is connective and further, if $X$ is $m$-connected, then so is $X \otimes_{{\text{Or}G}} \textbf{A}$. As a special case we deduce the following useful lemma.

\begin{lemma} \label{homologypreservesconnectivity}
Let $X$ be a $G$-CW-complex and $\textbf{A}$ a functor $\text{Or}G \rightarrow \text{Sp}_{\geq 0}$ with values in connective spectra. Denote by $X^{(k)}$ the $k$-skeleton of $X$. Then the homomorphism
$ \pi_n ( X^{(k)} \otimes_{{\text{Or}G}} \textbf{A} ) \rightarrow \pi_n ( X \otimes_{{\text{Or}G}} \textbf{A}) $
is an isomorphism for $n < k$ and an epimorphism for $n = k$.
\end{lemma}

\begin{proof} Consider first the inclusion $ X^{(k)} \rightarrow X^{(k+1)}$ of skeleta for some $k$. This map is given by a pushout square
$$\xymatrix{
\coprod_{i \in I} G/H_i \times S^k \ar[r]^(.6){(f,g)} \ar[d] & X^{(k)} \ar[d] \\
\coprod_{i \in I} G/H_i \times D^{k+1} \ar[r]   &  X^{(k+1)}.
}$$
Applying the colimit preserving functor $- \otimes_{{\text{Or}G}} \textbf{A}$ we obtain the pushout square of spectra
$$\xymatrix{
(\bigvee_{i \in I} \textbf{A}(G/H_i)) \otimes (\mathbb{S}^k \oplus \mathbb{S}^0) \ar[r]^(.6){(f,g)} \ar[d] & X^{(k)} \otimes_{{\text{Or}G}} \textbf{A} \ar[d] \\
\bigvee_{i \in I} \textbf{A}(G/H_i) \ar[r]   &  X^{(k+1)} \otimes_{{\text{Or}G}} \textbf{A}.
}$$
Since the vertical cofibers agree, we get the cofiber sequence
$$  X^{(k)} \wedge_{\text{Or}G} \textbf{A} \rightarrow X^{(k+1)} \wedge_{\text{Or}G} \textbf{A} \rightarrow \bigvee_{i \in I}  \textbf{A}(G/H_i) \otimes \mathbb{S}^{k+1}. $$
Since $\textbf{A}$ is object-wise connective, we get for $n < k+1$
$$ \pi_n( \textbf{A}(G/H_i) \otimes \mathbb{S}^{k+1} ) = \pi_{n-k-1}( \textbf{A}(G/H_i) ) = 0.$$
The claimed statements now follow by induction for $X$ being finite dimensional $G$-CW from the long exact sequence in homotopy groups of the above fiber sequence. For general $X$ we have $X = \text{colim}_k X^{(k)}$. The functor $ - \otimes_{{\text{Or}G}} \textbf{A}$ preserves colimits, thus $ \pi_n ( X \otimes_{{\text{Or}G}} \textbf{A} ) \cong \pi_n (  \text{colim}_k ( X^{(k)} \otimes_{{\text{Or}G}} \textbf{A} ) ) \cong \pi_n (   X^{(k)} \otimes_{{\text{Or}G}} \textbf{A}  )$ for $k>n$, thus reducing the lemma to the finite case.
\end{proof}

\begin{lemma} \label{1dimensionalclassifyingspace}
Suppose $\textbf{A}$ is a functor $\text{Or}G \rightarrow \text{Sp}$ and $X$ a $G$-CW-complex that admits a $1$-dimensional model of the form
$$\xymatrix{
\coprod_{i \in I} G/H_i \times S^0 \ar[r]^{(f,g)} \ar[d] & \coprod_{j \in J} G/K_i \ar[d] \\
\coprod_{i \in I} G/H_i \times D^1 \ar[r]   &  X
}$$
Then there is a fiber sequence
$$ \bigvee_{i \in I} \textbf{A}(G/H_i) \xrightarrow{f-g} \bigvee_{j \in J} \textbf{A}(G/K_j) \rightarrow X \otimes_{\text{Or}G} \textbf{A}. $$
\end{lemma}

\begin{proof}
Analogous to the proof of the previous lemma, the pushout square 
$$\xymatrix{
\coprod_{i \in I} G/H_i \times S^0 \ar[r]^{(f,g)} \ar[d] & \coprod_{j \in J} G/K_i \ar[d] \\
\coprod_{i \in I} G/H_i \times D^1 \ar[r]   &  X
}$$
produces the pushout square
$$\xymatrix{
\bigvee_{i\in I} \textbf{A}(G/H_i) \vee \textbf{A}(G/H_i) \ar[r]^{~~(f, g)} \ar[d]^{(id, id)} & \bigvee_{j\in J} \textbf{A}(G/K_j) \ar[d] \\
\bigvee_{i\in I} \textbf{A}(G/H_i) \ar[r] & X \otimes_{\text{Or}G} \textbf{A}.
}$$
of spectra. This is equivalent to the fiber sequence
$$ \bigvee_{i\in I} \textbf{A}(G/H_i) \vee \textbf{A}(G/H_i) \xrightarrow{\big(\begin{smallmatrix}
  \text{id} & \text{id}\\
  -f & - g
\end{smallmatrix}\big)}  \bigvee_{i\in I} \textbf{A}(G/H_i) \vee \bigvee_{j\in J} \textbf{A}(G/K_j) \rightarrow X \otimes_{\text{Or}G} \textbf{A}.$$
Elementary row and column reduction now yields the desired fiber sequence.
\end{proof}

\section{Assembly and the Farrell-Jones conjecture}

\begin{definition}
Let $G$ be a group. A \emph{family of subgroups} is a set of subgroups $\mathcal{F}$ that is closed under subgroups and conjugation.
\end{definition}

\begin{example}The following four examples of families will be relevant.
\begin{itemize}
\item The trivial family $\text{Triv}$ consisting of only the subgroup $\set{1}$.
\item The family $\text{All}$ consisting of all subgroups.
\item The family Fin consisting of all finite subgroups.
\item The family VCyc consisting of all virtually cyclic subgroups. A group $H$ is virtually cyclic if it contains a cyclic subgroup of finite index.
\end{itemize}
\end{example}

If $\textbf{A}$ is a functor $\text{Or}G$ to Sp, then
$$\text{colim}_{\text{Or}G} \textbf{A} \simeq \textbf{A}(G/G)$$
since $G/G \cong \text{pt}$ is a terminal object in the category $\text{Or}G$. The property of $\textbf{A}$ satisfying assembly states that this still holds true when the domain, over which the colimit is taken, is suitably restricted.

\begin{definition}
Let $\mathcal{F}$ be a family of subgroups for the group $G$ and $\textbf{A}$ a functor $\text{Or}G$ to Sp. Denote by $\text{Or}G_\mathcal{F}$ the full subcategory of $\text{Or}G$ spanned by the objects $G/H$ with $H \in \mathcal{F}$. Then the inclusion $\text{Or}G_\mathcal{F} \subset \text{Or}G$ induces a natural map
$$ \text{colim}_{\text{Or}G_\mathcal{F}} \textbf{A} \rightarrow \text{colim}_{\text{Or}G} \textbf{A} \simeq \textbf{A}(G/G). $$
We say $\textbf{A}$ \emph{satisfies assembly for} $\mathcal{F}$ if this map is an equivalence.
\end{definition}

\begin{lemma} \label{twooutofthreeassembly}
Assume we have a fiber sequence of functors $\text{Or}G \rightarrow$ Sp,
$$ \textbf{A} \implies \textbf{B} \implies \textbf{C}. $$
If any two of them satisfy assembly for a family $\mathcal{F}$, then so does the third.
\end{lemma}

\begin{proof}
Both $\text{colim}_{\text{Or}G_\mathcal{F}}$ as well as $\text{colim}_{\text{Or}G}$ are exact functors from the $\infty$-category $\text{Fun}( \text{Or}G, \text{Sp})$ to Sp giving the diagram
$$\xymatrix{
\text{colim}_{\text{Or}G_\mathcal{F}} \textbf{A} \ar[r] \ar[d] & \text{colim}_{\text{Or}G_\mathcal{F}} \textbf{B} \ar[r] \ar[d] & \text{colim}_{\text{Or}G_\mathcal{F}} \textbf{C} \ar[d] \\
\text{colim}_{\text{Or}G} \textbf{A} \ar[r] & \text{colim}_{\text{Or}G} \textbf{B} \ar[r] & \text{colim}_{\text{Or}G} \textbf{C}
}$$
with rows being fiber sequences. The statement now follows from the $5$-lemma.
\end{proof}

We are concerned with one particular type of functor on the orbit category - the functor that associates to $G/H$ the algebraic $K$-theory spectrum of its group algebra over a fixed base ring $R$, $\textbf{K}RH$. However, note that algebraic $K$-theory is a priori only functorial in ring homomorphisms. This models the morphisms $\textbf{K}RH \rightarrow \textbf{K}RH'$ corresponding to inclusions $H \subset H'$. We also need functoriality with respect to conjugation morphisms $c_g  \colon  G/H \rightarrow G/(g^{-1}Hg)$. These can give a priori different ring homomorphisms $RH \rightarrow Rg^{-1}Hg$ depending on the choice of representative $g$, meaning there is no good functor $\text{Or}G \rightarrow \text{Rings}$. This issue has been addressed by James Davis and Wolfgang Lück, and we will summarize the main results necessary for our work here. In the following, denote by $\textup{Grpds}_1$ the 1-category of groupoids and functors between them, and by $\textup{Grpds}_{(2,1)}$ the (2,1)-category of groupoids, functors and natural transformations between. We can interpret $\textup{Grpds}_{(2,1)}$ as an $\infty$-category, which is equivalent to the subcategory of Spc spanned by 1-truncated spaces. The category $\textup{Grpds}_1$ has a model structure for which the weak equivalences are given by equivalences of groupoids, and the $\infty$-category $\textup{Grpds}_{(2,1)}$ is obtained by inverting this class. \mbox{(See e.g. \cite[Theorem 5.4]{Hollander2008})}

\begin{lemma}[Davis, Lück \cite{Davis1998SpacesOA} Chapter 2 and Lemma 2.4] \label{davislück}
Let $G$ be a group and $R$ be a ring. There exists a functor of $1$-categories
$$ \textup{\textbf{K}}^{alg}R(-)  \colon  \textup{Grpds}_1 \rightarrow \Omega\textup{-Sp} $$
where $\Omega\textup{-Sp}$ is the $1$-category of $\Omega$-Spectra. It has the following properties:
\begin{enumerate}
\item If $F_i \colon \mathcal{G}_0 \rightarrow \mathcal{G}_1$ for $i = 0, 1$ are functors of groupoids and $T: F_0 \rightarrow F_1$ is a natural transformation between them, then the induced maps of spectra 
$$ \textup{\textbf{K}}^{alg}R ( F_i ) \colon \textup{\textbf{K}}^{alg}R (\mathcal{G}_0) \rightarrow \textup{\textbf{K}}^{alg}R (\mathcal{G}_1)$$
are homotopy equivalent.
\item Let $\mathcal{G}$ be a groupoid. Suppose that $\mathcal{G}$ is connected, i.e. there is a morphism between any two objects. For an object $x \in \text{Ob}(\mathcal{G})$, let $\mathcal{G}_x$ be the full subgroupoid with precisely one object, namely $x$. Then the inclusion $i_x : \mathcal{G}_x \rightarrow \mathcal{G}$ induces a homotopy equivalence 
$$\textup{\textbf{K}}^{alg}R(i_x) \colon \textup{\textbf{K}}^{alg}R (\mathcal{G}_x) \rightarrow \textup{\textbf{K}}^{alg}R(\mathcal{G})$$
and $\textup{\textbf{K}}^{alg}R(\mathcal{G}_x)$ is isomorphic to the non-connective algebraic $K$-theory spectrum associated to the group ring $R\text{aut}_\mathcal{G}(x)$. 
\end{enumerate}
\end{lemma}

The $\infty$-category of spectra is a localization of the nerve of the $1$-category of $\Omega$-spectra by the class of weak equivalences. Let $L : N( \Omega\text{-Sp} ) \rightarrow \text{Sp}$ be the corresponding localization functor. Lemma \ref{davislück} (1) implies that the composite $L \circ \textup{\textbf{K}}^{alg}R(-)$ sends equivalences of groupoids to equivalences in $\text{Sp}$ and thus descends to a unique, well-defined functor of $\infty$-categories
$$ \textup{\textbf{K}}^{alg}R(-)  \colon  \textup{Grpds}_{(2,1)} \rightarrow \textup{Sp},$$
which we will, by abuse of notation, denote by the same symbol.

\begin{remark}
A cleaner construction of $K$-theory as a functor of groupoids, that the author was not familiar with at the time of writing the first version of this paper, comes from using the non-connective $K$-theory functor on stable, idempotent complete $\infty$-categories
$$ \textbf{K} :  \mathrm{Cat}^{\mathrm{perf}} \rightarrow \mathrm{Sp} $$
as developed in \cite{Blumberg_2013}. The wanted functor defined on groupoids can be obtained by pre-composition with the functor
$$ \textup{Grpds}_{(2,1)} \rightarrow \mathrm{Cat}^{\mathrm{perf}},~ \mathcal{G} \mapsto \mathcal{G} \otimes \mathrm{Perf}_R, $$
where $\mathrm{Perf}_R$ is the stable, idempotent complete $\infty$-category of perfect $R$-complexes. This has the advantage that the ability to work with ring spectra instead of just rings is immediate. We will not digress further on this matter.
\end{remark}

\begin{theorem} \label{ktheoryfunctor}
Let $G$ be a group and $R$ be a ring. There exists a functor
$$ \textup{\textbf{K}}R(-)  \colon  \textup{Or}G \rightarrow \textup{Sp} $$
with the properties
\begin{enumerate}
\item $\textup{\textbf{K}}R(G/H) \simeq \textup{\textbf{K}}RH$ where $\textup{\textbf{K}}RH$ is the non-connective algebraic $K$-theory spectrum of the group ring $RG$.
\item If $H \subset H'$ giving the canonical map $G/H \rightarrow G/H'$, the induced map $\textup{\textbf{K}}RH \rightarrow \textup{\textbf{K}}RH'$ corresponds to the map induced by the ring homomorphism $RH \rightarrow RH'$.
\item The induced $G$-action on the spectrum $\textup{\textbf{K}}R$ via the conjugation morphisms $c_g  \colon  G/\set{1} \rightarrow G/\set{1}$ is trivial.
\item Let $g \in G$. The action of the conjugation morphism $c_g  \colon  G/H \rightarrow G/(g^{-1}Hg)$ on the homotopy groups $K_n RH \rightarrow K_n R (g^{-1}Hg)$ is induced by the ring homomorphism $RH \rightarrow R (g^{-1}Hg), h \mapsto g^{-1}hg$ and independent of the chosen representative $g$.
\end{enumerate}
\end{theorem}

\begin{remark}
The action in the last point of $c_g$ on $K_0$ can be understood in the following way. It sends a f.g.\ projective $RH$-module $P$ to the $R (g^{-1}Hg)$-module $P^g$ with same underlying $R$-module as $P$ and the scalar multiplication of an element $h' \in g^{-1}Hg$ on $x \in P$ given by $h' \ast x \defeq g h' g^{-1} \cdot x$.
\end{remark}

\begin{proof} A few words before we begin with the actual proof: In using Lemma \ref{davislück} we never really need to know what the value of $\textbf{K}^{alg}R(\mathcal{G})$ at some groupoid $\mathcal{G}$ actually is, only that a \emph{choice} of base-point $x \in \mathcal{G}$ induces an equivalence $\textbf{K}^{alg}R(\mathcal{G}) \simeq \textup{\textbf{K}}R(G/H)$. Different choices lead to (non-trivial) maps of $K$-theory spectra.

There is a functor $G \smallint ( - ) : \textup{Or}G \rightarrow \textup{Grpds}_{(2,1)}$ which sends the $G$-set $G/H$ to the groupoid $G \smallint G/H$ with object set $G/H$ and a morphism from $gH$ to $g'H$ for each element $g'' \in G$ such that $g''gH = g'H$.\footnote{This provides an actual functor of $1$-categories. Then post-compose with the localization $\textup{Grpds}_{1} \rightarrow \textup{Grpds}_{(2,1)}$. } Define $\textup{\textbf{K}}R(-)$ as the composite 
$$ \textup{\textbf{K}}R(-) \defeq \textbf{K}^{alg}R \circ G \smallint ( - ) : \textup{Or}G \rightarrow \textup{Sp} $$
with the functor $\textbf{K}^{alg}R$ being given as above.

For a subgroup $H$ of $G$ write $BH$ for the groupoid with a single object and automorphism group $H$. 
The statement that $\textup{\textbf{K}}R(G/H) \simeq \textup{\textbf{K}}RH$ follows from Lemma \ref{davislück} (2), since the automorphism group of the object $H$ in $G \smallint G/H$ is exactly $H$. This means we have the inclusion functor $BH \rightarrow G \smallint G/H$, which produces the claimed equivalence in the first point.

If $H \subset H'$ are two subgroups of $G$, we have a functor $BH \rightarrow BH'$, which fits into the commutative square
$$\xymatrix{
BH \ar[r] \ar[d] & G \smallint G/H \ar[d] \\
BH' \ar[r] & G \smallint G/H'.
}$$
The map $\textbf{K}^{alg}R(BH) \rightarrow \textbf{K}^{alg}R(BH')$ is equivalent to the map $\textup{\textbf{K}}RH \rightarrow \textup{\textbf{K}}RH'$ in algebraic $K$-theory induced by the ring homomorphism $RH \rightarrow RH'$, proving the second claim.

For the third claim note that the groupoid $G \smallint G/\set{1}$ is contractible since every object is in fact both terminal and initial, which means the induced $G$-action on $G \smallint G/\set{1} \simeq \mathrm{pt}$ is trivial. In particular, the same holds when applying $K$-theory.

For the last point, we have a commutative square of groupoids
\begin{center}
\begin{tikzcd}
BH \arrow[r,"{\text{inc}_H}"] \arrow[d,swap,"{ g^{-1}(-)g }"] & G \smallint G/H \ar[d,"{c_g  \circ \text{trans}_{g^{-1}}}"] \\
B(g^{-1}Hg) \arrow[r,"{\text{inc}_{g^{-1}Hg}}"] & G \smallint G/(g^{-1}Hg) 
\end{tikzcd}
\end{center}
where the functor $g^{-1}(-)g: BH \rightarrow B(g^{-1}Hg)$ is given via the group homomorphism that is the conjugation $h \mapsto g^{-1}hg$ (it is trivial on objects as both groupoids have only a single object), and 
$$\text{trans}_{g^{-1}} : G \smallint G/H \rightarrow G \smallint G/H$$ 
is the functor that acts on objects as 
$$g' H \mapsto g^{-1} g' H,$$
and sends the morphism 
$$g' H \xrightarrow{k} k g' H$$ to 
$$g^{-1} g' H \xrightarrow{ g^{-1}k g} g^{-1} k  g' H.$$ 
Note that $\text{trans}_{g^{-1}}$ is an auto-equivalence of $G \smallint G/H$ that sends the object $ H $ to $g^{-1} H$.
The group homomorphism $c_g \colon H \rightarrow g^{-1}Hg$ induces the claimed map in $K_n$, thus showing the last point. 
If $g''$ is another element in $G$ such that $c_g = c_{g''} : G/H \rightarrow G/g^{-1}Hg$ represent the same map in $\text{Or}G$, which is exactly the case when $g (g'')^{-1} \in H$, it is elementary to show that $\text{trans}_{g^{-1}}$ and $\text{trans}_{(g'')^{-1}}$ are naturally equivalent functors, thus inducing homotopic maps on $K$-theory spectra and therefore the same map in $K_n$.
\end{proof}

\begin{remark}
The reader may be tempted to wonder what the action of $\text{trans}_{g}$ on $\textup{\textbf{K}}RH$ under the identification of $\textup{\textbf{K}}RH$ with $\textbf{K}^{alg}R(G \smallint G/H)$ is. Such a question carries some risk of confusion, as this identification hinges on the choice of basepoint $H$ in $G \smallint G/H$, which is not preserved by $\text{trans}_{g}$, unless $g \in H$. While the automorphism group of the point $gH$ is of course isomorphic to $H$, this identification is again obtained by conjugation with some element, so the answer simply reduces to the already established identification of the conjugation action.
\end{remark}

\begin{definition}[Farrell-Jones conjecture] A group $G$ has the \emph{Farrell-Jones property} if the functor for non-connective algebraic $K$-theory $\textbf{K}R  \colon  \text{Or}G \rightarrow \text{Sp}$ satisfies assembly for the family of virtually cyclic subgroups and any ring $R$. We will also sometimes refer to this as saying that $G$ \emph{satisfies the Farrell-Jones conjecture}.
\end{definition}

\begin{remark}The Farrell-Jones conjecture is known to hold for a wide range of groups. A recent summary of results can be found in \cite{reich_varisko}, Theorem 27.
\end{remark}

In order to compute the colimits involved in the assembly maps, a useful tool for geometric arguments is the notion of classifying spaces for a family of subgroups $\mathcal{F}$ of $G$.

\begin{definition}
Let $\mathcal{F}$ be a family of subgroups of $G$. We call the $G$-space $E(G;\mathcal{F})$ with
$$ E(G;\mathcal{F})^H \simeq \begin{cases} \text{pt} & \text{if}~ H \in \mathcal{F} \\ \emptyset & \text{if}~ H \not\in \mathcal{F}, \end{cases} $$
the \emph{classifying space for the family} $\mathcal{F}$. Note that $E(G;\mathcal{F})$ is unique up to equivalence. We will sometimes write $E\mathcal{F}$ if the context is clear.
\end{definition}

\begin{remark}
The $G$-space $E(G;\mathcal{F})$ can always be modelled by a $G$-CW-complex, see \cite{Lueck2005SurveyOC} Theorem 1.9.
\end{remark}

\begin{example}
The universal cover $EG$ of $BG$ with its free $G$-action is a classifying space for the trivial family $\text{Triv}$ consisting only of the single subgroup $\set{1}$. The point $\text{pt} \simeq G/G$ with trivial $G$-action is a classifying space for the family $\text{All}$ of all subgroups.
\end{example}

\begin{theorem}[See also \cite{Mathew_2019}, Proposition A.2]
Let $i_\mathcal{F}$ be the inclusion of the category $\textup{Or}G_\mathcal{F}$ into the category of $G$-spaces. Then
$$ \textup{colim}(i_\mathcal{F}) \simeq E(G;\mathcal{F}).$$
\end{theorem}

\begin{proof}
Let $H$ be a subgroup of $G$. Under the equivalence $G$-Spc $\simeq \text{Fun}( \text{Or}G^{op}, \text{Spc} )$ given by Elmendorf's theorem the operation of taking $H$-fixed points corresponds to evaluation at $G/H$. Since colimits of functors are computed objectwise this means that taking $H$-fixed points commutes with colimits. Now, if $H \not\in \mathcal{F}$, then 
$$\text{colim}(i_\mathcal{F})^H \simeq \text{colim}_{G/K \in \text{Or}G_\mathcal{F}}( (G/K)^H ) = \emptyset.$$
Now suppose $H$ in $\mathcal{F}$. Then
$$\text{colim}(i_\mathcal{F})^H \simeq \text{colim}_{G/K \in \text{Or}G_\mathcal{F}}( \text{Map}_{\text{Or}G_\mathcal{F}}(G/H, G/K) )$$
is the colimit over a corepresentable functor and thus contractible.
\end{proof}

\begin{remark}
The statement that a colimit over a corepresentable functor is contractible can be seen in the following way. If $c \in \mathcal{C}$ is an object, then 
$$\mathrm{colim}_{\mathcal{C}} \mathrm{Map}_{\mathcal{C}}(c,-) \simeq | \mathcal{C}_{c/} | \simeq \mathrm{pt}.$$ 
Here we use the fact that the colimit of  $\mathrm{Spc}$-valued functor is given as the realization of its classifying left fibration (see e.g. \cite{luriehtt}, Corollary 3.3.4.6.), and that the classifying left fibration of $\mathrm{Map}_{\mathcal{C}}(c,-)$ is the forget functor $\mathcal{C}_{c/} \rightarrow \mathcal{C}$ (see e.g. \cite{LandIntroductionInfinityCategories}, Chapter 4.2.). The $\infty$-category $\mathcal{C}_{c/}$ is evidently contractible, as it contains an initial object.
\end{remark}

The following lemma now explains why classifying spaces of families are such a useful tool for understanding assembly maps.
\begin{lemma}
Let $\mathcal F$ be a family of subgroups of $G$ and $\textbf{A}$ a functor $\text{Or}G \rightarrow \text{Sp}$. There is a natural equivalence
$$ \text{colim}_{\text{Or}G_\mathcal{F}} \textbf{A} \xrightarrow{\sim} E(G;\mathcal{F}) \otimes_{\text{Or}G} \textbf{A}. $$
\end{lemma}

\begin{proof} This follows from the definition of the orbit tensor product. We have
$$E(G;\mathcal{F}) \otimes_{\text{Or}G} \textbf{A} \simeq ( \text{colim}_{G/K \in \text{Or}G_\mathcal{F}}( G/K ) )\otimes_{\text{Or}G} \textbf{A} \simeq \text{colim}_{G/K \in \text{Or}G_\mathcal{F}}( \textbf{A}(K) ). $$
\end{proof}
We will also need the following two results on assembly in $K$-theory.
\begin{theorem}[See \cite{Bartels_2003}] \label{fintovcycissplit}
The map
$$ E\textup{Fin} \otimes_{\textup{Or}G} \textup{\textbf{K}} R \rightarrow E\textup{VCyc}\otimes_{\text{Or}G} \textup{\textbf{K}} R $$
is split injective and is so naturally with respect to the ring $R$ and the group $G$.
\end{theorem}

\begin{lemma}[See \cite{Lueck_Reich_2005}, Proposition 2.14]  \label{fintovcycisorational}  
Let $k$ be a field of characteristic $0$. Then the map
$$ E \text{Fin}\otimes_{\text{Or}G} \textbf{K} k   \rightarrow E \text{VCyc}\otimes_{\text{Or}G} \textbf{K} k $$
is an equivalence.
\end{lemma}

\begin{remark} \label{rationalktheoryfiniteassembly}
A consequence of Lemma \ref{fintovcycisorational} is that if $G$ satisfies the Farrell-Jones conjecture, the functor $\textbf{K} k -$ actually satisfies finite assembly. Since $\textbf{K} k H$ is connective for finite groups $H$, a fact that will be discussed later at the beginning of Section \ref{lowerktheory}, see Proposition \ref{negativektheoryfield}, this implies that  $K_0 k-$ satisfies finite assembly too, in the sense that 
$$K_0 k G = \text{colim}_{G/H \in \text{Or}G_\text{Fin}} K_0 k H$$
with the colimit in question being relative to the $1$-category of abelian groups. This is because the colimit
$$ \text{colim}_{G/H \in \text{Or}G_\text{Fin}} \textbf{K} k H \simeq \textbf{K} k G $$
agrees with the colimit taken in connective spectra $\text{Sp}_{\geq 0}$, as $\text{Sp}_{\geq 0}$ is closed under colimits in $\text{Sp}$, and $\pi_0$ is left adjoint to the inclusion
$ \text{Ab} \rightarrow \text{Sp}_{\geq 0}.$
\end{remark}

\section{The Whitehead spectrum $\textbf{Wh}(R; G)$ and the spectrum $\textbf{SC}(G)$} \label{whiteheadspectrum}

We remind the reader that for any ring $R$, $\widetilde{K_{0}}R$ is defined as the cokernel of the natural homomorphism
$ K_0 \mathbb{Z} \rightarrow K_0 R$. For a group ring $RG$ the group $K_0 RG$ naturally has $K_0 R$ as a split summand, with the split given via the augmentation map $RG \rightarrow R$ that sends all $g \in G$ to $1$.
If the base ring $R$ is such that every projective module is stably free, such as when $R$ is a PID or a local ring, it follows that $K_0 R \cong K_0 \mathbb{Z} \cong \mathbb{Z}$, and we have $K_0 RG \cong \mathbb{Z} \oplus \widetilde{ K_0 }RG$. 

We are ultimately interested in understanding the map
$$ \widetilde{K_0 }\mathbb{Z}G \rightarrow \widetilde{K_0}\mathbb{Q}G.$$
Since our tool of choice, the Farrell-Jones conjecture, gives us a priori the full map on spectra $\textbf{K} \mathbb{Z}G \rightarrow \textbf{K} \mathbb{Q}G$, we would like to split off the superfluous data in a sensible way. This is where the Whitehead spectrum comes into play. Recall that $EG$ is the classifying space for the trivial family. We have the equivalence
$$ EG \otimes_{\text{Or} G} \textbf{K}R \simeq \text{colim}_{BG} \textbf{K}R \simeq BG \otimes \textbf{K}R $$
where we used that the subcategory of $\text{Or}G$ generated by the single object $G/\set{1}$ is an $\infty$-groupoid equivalent to $BG$ and that the action of $G$ on the value $\textbf{K}R(G/\set{1}) = \textbf{K}R$ is trivial by Theorem \ref{ktheoryfunctor} $(3)$. The natural map $$ EG \rightarrow E(G;\text{All}) \simeq G/G$$ induces the assembly map
$$ BG \otimes \textbf{K}R \simeq EG \otimes_{\text{Or} G} \textbf{K}R \rightarrow G/G \otimes_{\text{Or} G} \textbf{K}R = \textbf{K}R[G].$$

\begin{definition} Given a group $G$ and a ring $R$, we define the \emph{Whitehead spectrum} $\textbf{Wh}(R; G)$ to be the cofiber of the assembly map
$$ BG \otimes \textbf{K}R \rightarrow \textbf{K}R[G] \rightarrow \textbf{Wh}(R; G), $$
corresponding to the trivial family $\text{Triv}$.
\end{definition}

\begin{example}
Let $R$ be a ring. Then:
\begin{itemize}
\item $\pi_{-i} \textbf{Wh}(R; G) = K_{-i}RG$ for $i > 0$, if $R$ is regular noetherian (see \cite{Weibel2013TheKA} III, Definition 4.1.),
\item $\pi_{0} \textbf{Wh}(R; G) = \widetilde{K_0 }R G,$ if $R$ is in addition a local ring or a PID (see \cite{Weibel2013TheKA} II, §2),
\item and furthermore, if $R$ is  a field or the ring of integers, then 
$$\pi_{1} \textbf{Wh}(R;G) = K_1 (RG) / \left\lbrace rg | r \in R^\times, g \in G \right\rbrace. $$
\end{itemize}

In the particular case of $R$ being the integers we have:
\begin{itemize}
\item $\pi_{-1} \textbf{Wh}(\mathbb{Z}; G) = K_{-1} \mathbb{Z} G,$
\item $\pi_{0} \textbf{Wh}(\mathbb{Z}; G) = \widetilde{K_0 }\mathbb{Z} G,$
\item $\pi_{1} \textbf{Wh}(\mathbb{Z}; G) = \text{Wh}(G) = K_{1} \mathbb{Z} G /( \left\lbrace \pm 1 \right\rbrace \otimes G_{ab} ),$ with $\text{Wh}(G)$ being the Whitehead group of $G$.
\end{itemize}
\end{example}

More generally, we can do the following construction.  The functor
$$ (B(-) \otimes \textbf{K}R )(G/H) \defeq (G/H \times EG) \otimes_{\text{Or} G} \textbf{K}R $$
from $\text{Or}G$ to spectra comes with a natural transformation 
$$\theta  \colon  B(-) \otimes \textbf{K}R \rightarrow \textbf{K}R(-), $$
induced from the projection $G/H \times EG \rightarrow G/H$, to the functor $\textbf{K}R$. Let us explain why the notation $B(-) \otimes \textbf{K}R$ makes sense. The value at $G/G$ is given as
$$ EG \otimes_{\text{Or} G} \textbf{K}R \simeq BG \otimes \textbf{K}R. $$
This means that the natural transformation $\theta$ becomes the assembly map
$$ BG \otimes \textbf{K}R \rightarrow \textbf{K}RG $$
when evaluated at $G/G$. If $G/H$ is an arbitrary object of $\text{Or}G$, then $G/H \times E(G;\text{Triv}) \simeq \text{Ind}_H^G(E(H;\text{Triv}))$. We can use the projection formula, Proposition \ref{projectionformula}, to get
\begin{align*}
(G/H \times EG) \otimes_{\text{Or} G} \textbf{K}R &= \text{Ind}_H^G(E(H;\text{Triv})) \otimes_{\text{Or} G} \textbf{K}R \\
 & \simeq E(H;\text{Triv}) \otimes_{\text{Or}H} \textbf{K}R \simeq BH \otimes \textbf{K}R.
\end{align*}
and see that $\theta$ has as component on the object $G/H$ the assembly map
$$BH \otimes \textbf{K}R \rightarrow \textbf{K}RH. $$
This allows us to define the functor $\textbf{Wh}(R; -)  \colon  \text{Or}G \rightarrow \text{Sp}$ as the cofiber of this natural transformation.

\begin{remark}
More generally, if $\mathcal{F}$ is a family and $\textbf{A}  \colon  \text{Or}G \rightarrow \text{Sp}$ is a functor, we can define
$$\textbf{A}_\mathcal{F}(G/H) \defeq (G/H \times E\mathcal{F}) \otimes_{\text{Or} G} \textbf{A} $$
and get via the projection $G/H \times E\mathcal{F} \rightarrow G/H$ a natural transformation
$$\textbf{A}_\mathcal{F} \rightarrow \textbf{A}.$$
The functor $\textbf{A}_\mathcal{F}$ can be shown to satisfy $\mathcal{F}$-assembly, and we can think of it as a universal approximation of $\textbf{A}$ from the left by a functor that satisfies $\mathcal{F}$-assembly. This construction appears for example in \cite{Davis_Quinn_Reich_2011}, Lemma 4.1. 
\end{remark}

The following is an essential lemma that is a consequence of Theorem \ref{fintovcycisorational} and Lemma \ref{fintovcycissplit}.

\begin{lemma} \label{whiteheadsplitinj}
The map
$$ E\text{Fin} \otimes_{\text{Or}G} \textbf{Wh}(R;-) \rightarrow E\text{VCyc} \otimes_{\text{Or}G} \textbf{Wh}(R;-) $$
is split injective and is so naturally with respect to the ring $R$ and the group $G$.
\end{lemma}

\begin{proof}
We have the commutative diagram
$$\xymatrix{
BG \otimes \textbf{K}R \ar[r]^= \ar[d] & BG \otimes \textbf{K}R \ar[d] \\
E\text{Fin} \otimes_{\text{Or}G} \textbf{K}R \ar[r] \ar[d] & E\text{VCyc} \otimes_{\text{Or}G} \textbf{K}R \ar[d] \\
 E\text{Fin} \otimes_{\text{Or}G} \textbf{Wh}(R;-) \ar[r] & E\text{VCyc} \otimes_{\text{Or}G} \textbf{Wh}(R;-)
}$$
with the columns being fiber sequences, so a natural split in the middle map induces one on the bottom. Hence, the statement follows from theorem \ref{fintovcycissplit}.
\end{proof}

An immediate consequence of Lemma \ref{twooutofthreeassembly} is that if $G$ satisfies the Farrell-Jones conjecture, the functor $\textbf{Wh}(R;-)$ on $\text{Or}G$ satisfies assembly for the family VCyc. moreover, Theorem \ref{fintovcycisorational} implies that for $R = \mathbb{Q}$, the functor $\textbf{Wh}(\mathbb{Q};-)$ then also satisfies assembly for the family Fin.

\begin{corollary} \label{plocallemma}
If $G$ satisfies the Farrell-Jones conjecture, then the image of the map
$$\widetilde{K_0 }\mathbb{Z} G \rightarrow \widetilde{K_0 }\mathbb{Q} G $$
agrees with the image of the map
\begin{equation*} \pi_0( E \text{Fin} \otimes_{\text{Or}G} \textbf{Wh}(\mathbb{Z},- )) \rightarrow \pi_0( E \text{Fin} \otimes_{\text{Or}G} \textbf{Wh}(\mathbb{Q},- ) ) \cong \widetilde{K_0 }\mathbb{Q} G. \tag{$\star$}
\end{equation*}
In particular if $( \star )$ vanishes $p$-locally for some prime $p$, then so does the map
$$ \widetilde{K_0 }\mathbb{Z} G \rightarrow \widetilde{K_0 }\mathbb{Q} G.$$
\end{corollary}

\begin{proof}
By Lemma \ref{whiteheadsplitinj} the group $\pi_0( E \text{Fin} \otimes_{\text{Or}G} \textbf{Wh}(\mathbb{Z},- ))$ is a split summand of $\widetilde{K_0 }\mathbb{Z} G$, similarly for $\mathbb{Q}$, so by naturality of the split with respect to ring homomorphisms, the map $ \widetilde{K_0 }\mathbb{Z} G \rightarrow \widetilde{K_0 }\mathbb{Q} G$ splits as a sum of two maps, the second of which has to be trivial, since $\pi_0( E \text{Fin} \otimes_{\text{Or}G} \textbf{Wh}(\mathbb{Q},- ) ) \cong \widetilde{K_0 }\mathbb{Q} G $. To get the second statement, apply the functor $\mathbb{Z}_{(p)}\otimes -$ and use exactness.
\end{proof}

\begin{definition} \label{singularcharactersdefinition}
Define the \emph{spectrum of singular characters} $\textbf{SC}(G)$ as the cofiber
$$ \textbf{Wh}(\mathbb{Z};G ) \rightarrow \textbf{Wh}(\mathbb{Q};G ) \rightarrow \textbf{SC}(G).$$
Write in short $\text{SC}(G) \defeq \pi_0 \textbf{SC}(G).$
\end{definition}

Note that we always have a long exact sequence
$$\begin{tikzcd}
\cdots \arrow[r] & \text{Wh}(G)  \arrow[r] & K_1 ( \mathbb Q G) / \left\lbrace rg | r \in \mathbb Q^\times, g \in G \right\rbrace \arrow[r]
\arrow[d, phantom, ""{coordinate, name=Z}]
& \pi_1 \textbf{SC}(G) \arrow[dll,rounded corners=8pt,curvarr=Z] 
\\
 & \widetilde{K_0 }\mathbb{Z} G  \arrow[r]
&  \widetilde{K_0 }\mathbb{Q} G \arrow[r]\arrow[d, phantom, ""{coordinate, name=W}]
& \text{SC}(G) 
\arrow[dll,rounded corners=8pt,curvarr=W] 
\\
 & \ K_{-1} \mathbb{Z} G \arrow[r] & \cdots & {}
\end{tikzcd}$$

From this we see that the vanishing of the map $\widetilde{K_0 }\mathbb{Z} G \rightarrow \widetilde{K_0 }\mathbb{Q} G$ is equivalent to the injectivity of $\widetilde{K_0 }\mathbb{Q} G \rightarrow \text{SC}(G)$. We will give a concrete description of the group $\text{SC}(G)$ for finite groups in the next section.

\section{Lower $K$-theory of finite groups} \label{lowerktheory}

The following section will use basics of representation theory, as for example treated in \cite{serre1977linear}. In this section, we will assume $G$ is finite. We  write $\mathbb{Z}_p$ and $\mathbb{Q}_p$ for the $p$-adic integers as well as $p$-adic rationals for a prime $p$. We will be concerned with the groups $K_0 \mathbb{Z}_p G, K_0 \mathbb{Q} G, K_0 \mathbb{Q}_p G$, as well as $K_{-1} \mathbb{Z} G$.

Suppose in the following that $k$ is a subfield of $\mathbb{C}$. It is a standard fact in representation theory that the group ring $kG$ is semisimple. In particular, it decomposes uniquely as
$$ kG \cong \prod_{I \in \text{Irr}_k(G) } M_{n_I \times n_I} (D_I) $$
where $\text{Irr}_k(G)$ is the set of isomorphism classes of irreducible $k$-representations of $G$ and $D_I$ are division algebras given by
$$ D_I = \text{hom}_{G} ( I , I ), $$
and $n_I = \left\langle kG , I \right\rangle$ is the multiplicity of $I$ appearing in the regular representation $kG$. This is known as the \emph{Wedderburn decomposition} of $kG$.
The $K$-theory of a division algebra $D$ is $\mathbb{Z}$ in degree $0$ since every left $D$-module is in fact a $D$-vectorspace and its negative $K$-theory vanishes since $D$ is regular noetherian. Using that $K$-theory commutes with products as well as invariance of $K$-theory under Morita-equivalence, we get the formula $K_0 kG \cong \mathbb{Z}^{r_k(G)}$ where $r_k(G)$ is the number of irreducible $k$-representations of $G$ and the irreducible representations form a set of generators for this group. We also get that $K_{-n} kG = 0$ for all $n>0$.

In the following, if $D$ is a division algebra over a field $F$, we denote by $[D,F]$ the dimension of $D$ as an $F$-vector space.

\begin{definition}[Schur index, see also \cite{dieck2006transformation} 9.3. and \cite{serre1977linear} 12.2.]
Let $I$ be an irreducible $k$-representation of $G$. Then $D_I$ is a division algebra over its center $K_I$ of degree $m(I)^2$ with $m(I) = [D_I, E_I]$ for $E_I$ a maximal field contained in $D_I$. We call $m(I)$ the \emph{Schur index} of $I$.
\end{definition}

Assume $I$ is a $k$-representation of $G$. Recall that the \emph{character} of $I$ is defined as the function
\begin{align*}
\chi_I \colon \text{Conj(G)} &\rightarrow k \\
g &\mapsto \text{tr}( g\cdot  \colon I \rightarrow I ).
\end{align*}
This is well-defined since the trace of an endomorphism is invariant under conjugation. We call a character irreducible if it is obtained as the character of an irreducible representation. For the following Lemma, we emphasize that we use a fixed embedding $k \subset \mathbb{C}$.

\begin{lemma}[\cite{serre1977linear} 12.2., also \cite{isaacs1976character} Corollary 10.2] \label{galoistranslates}
Let $\chi$ be an irreducible complex character. There exists an irreducible $k$-representation $I$, such that 
$$\chi_I = m(I) \sum_\rho \rho(\chi), $$ with the $\rho(\chi)$ being all the distinct translates of $\chi$ under the action of the Galois group $\text{Gal}( \mathbb{C}/ k )$, in other words the sum is over the orbit of the action of the Galois group $\text{Gal}( \mathbb{C}/ k )$ on the set of characters. Conversely, if $I$ is an irreducible $k$-representation, its character splits as above over $\mathbb{C}$.
\end{lemma}

\begin{corollary} \label{splittingofplocalization}
The map $K_0 \mathbb{Q} G \rightarrow K_0 \mathbb{Q}_p G$ is acting on the basis of irreducible $\mathbb{Q}$-representations by
$$[I] \mapsto \frac{m(I)}{m(K_i)} ([K_1] + \cdots + [K_{n_I}] ).$$ 
where the $K_i$ are representatives of the irreducible $\mathbb{Q}_p$-representations that appear in $I \otimes \mathbb{Q}_p$ and $n_I$ depends on $I$. The Schur index $m(K_i)$ is independent of $i$.
If $I$ and $J$ are distinct irreducible $\mathbb{Q}$-representations, the irreducible components appearing in their individual $p$-completions are pairwise non-isomorphic. In other words, the map
$$K_0 \mathbb{Q} G \rightarrow K_0 \mathbb{Q}_p G$$
splits as
$$\bigoplus_{ I \in \text{Irr}_\mathbb{Q}(G)} \mathbb{Z} \rightarrow \bigoplus_{ I \in \text{Irr}_\mathbb{Q}(G)} \mathbb{Z} \set{ K \in \text{Irr}_{\mathbb{Q}_p}(G) | K ~\text{appears as a summand in}~ I\otimes \mathbb{Q}_p }. $$
\end{corollary}

\begin{proof}
Split the character $\chi_I$ as in Lemma \ref{galoistranslates} for the case $k = \mathbb{Q}$, then apply the Lemma \ref{galoistranslates} to each of the irreducible constituents appearing, using the case $k = \mathbb{Q}_p$.
\end{proof}

\begin{definition}[Local Schur index]
Let $I$ be an irreducible $\mathbb{Q}$-representation of $G$. The \emph{local Schur index} $m_p(I)$ of $I$ at the prime $p$ is defined to be the Schur index of any of the irreducible components of $I \otimes \mathbb{Q}_p$ and is independent of this choice by the argument given above. The \emph{local Schur index at infinity} $m_\infty (I)$ is similarly defined as the Schur index of any of the irreducible components of $I \otimes \mathbb{R}$. 
\footnote{As a good summary of what is currently known about rational and local Schur indices we recommend \cite{Unger_2019}.}
\end{definition}

\subsection{Negative $K$-theory of finite groups}

We are now concerned with collecting results on the groups $K_{-i} R G$ for a ring $R$.

\begin{proposition} \label{negativektheoryfield}
Let $G$ be a finite group and let $k$ be a field of characteristic $0$. Then $K_{-i} k G$ vanishes for $i > 0$.
\end{proposition}
We note that this implies in particular that the spectrum $\textbf{Wh}( k; G )$ is connective, as it has the same negative homotopy groups as $\textbf{K} k G$.
\begin{proof}
This is immediate from the decomposition
$$ kG \cong \prod_{I \in \text{Irr}_k(G) } M_{n_I \times n_I} (D_I) $$
with division algebras $D_I$, together with the fact that negative $K$-theory of (left) regular Noetherian rings (such as division algebras) vanishes, \cite{weibel1995introduction} III.4.
\end{proof}
Next we shall be concerned with the case $R = \mathbb{Z}$. The following result on negative $K$-theory of group rings is due to Carter.

\begin{theorem}[Carter, \cite{carternegktheory}] \label{negativektheory}
Let $G$ be finite. The groups $K_{-i} \mathbb{Z} G$ vanish for $i > 1$ and the group $K_{-1} \mathbb{Z} G$ has the form
$$K_{-1} \mathbb{Z} G = \mathbb{Z}^r \oplus (\mathbb{Z}/2)^s$$
where
$$ r = 1 - r_\mathbb{Q} + \sum_{p | ~ |G|}( r_{\mathbb{Q}_p} - r_{\mathbb{F}_p})$$
and $s$ is equal to the number of irreducible $\mathbb{Q}$-representations $I$ with even Schur index $m(I)$ but odd local Schur index $m_p(I)$ at every prime $p$ dividing the order of $G$. 
\end{theorem}

We note that it can be shown that the smallest group $G$ such that $s>0$ is the group $Q_{16}$. Its negative $K$-theory will be computed in Section \ref{generalizedquaternion16}.

\begin{remark} \label{pgroupshaveonlytorsion}
If $G$ is a $p$-group, then by Magurn, \cite{magurn_negative} Theorem 1, the rank $r = 0$, in other words $K_{-1} \mathbb{Z} G$ only consists of $2$-torsion.
\end{remark}

\subsection{Localization squares for finite groups}

In the following, we will reinterpret some of the techniques that Carter used to prove Theorem \ref{negativektheory} in a more modern light. First, we need the following lemma, which is a consequence of a theorem due to Karoubi, see e.g. \cite{Weibel2013TheKA}, Prop. V.7.5. and Example V.7.5.1.

\begin{lemma}  \label{karoubisquare}
Let $G$ be a finite group and $P$ a finite set of primes. Then the following square is a pullback square of spectra:
$$ 
\xymatrix{ 
\textbf{K} \mathbb{Z} G \ar[r] \ar[d] & \bigvee_{p \in P} \textbf{K} \mathbb{Z}_p G \ar[d] \\
\textbf{K} \mathbb{Z}[P^{-1}] G \ar[r] & \bigvee_{p \in P} \textbf{K} \mathbb{Q}_p G,
}$$
where the maps appearing are induced by the corresponding inclusions of the involved rings.
\end{lemma}

\begin{corollary} \label{whiteheadlocalization}
Fix the same assumptions as in the previous lemma. Then
$$ 
\xymatrix{
\textbf{Wh}( \mathbb{Z}; G) \ar[r] \ar[d] & \bigvee_{p \in P} \textbf{Wh}( \mathbb{Z}_p; G ) \ar[d] \\
\textbf{Wh}( \mathbb{Z}[P^{-1}]; G ) \ar[r] & \bigvee_{p \in P} \textbf{Wh}( \mathbb{Q}_p; G )
}$$
is a pullback square of spectra as well.
\end{corollary}

\begin{proof}
The square in question is given as the levelwise cofibers of the square
$$\xymatrix{
BG \otimes \textbf{K} \mathbb{Z} \ar[r] \ar[d] 			& \bigvee_{p \in P} BG \otimes \textbf{K} \mathbb{Z}_p \ar[d] 	\\
BG \otimes \textbf{K} \mathbb{Z}[P^{-1}] \ar[r] 			& \bigvee_{p \in P} BG \otimes \textbf{K} \mathbb{Q}_p 			
}$$
and the square in Lemma \ref{karoubisquare}. Both are pushouts via the previous lemma, hence the claim follows.
\end{proof}

The following is a straightforward consequence of the fact that any idempotent in the rational group algebra $\mathbb{Q}G$ is already defined over $\mathbb{Z}[P^{-1}]G$.

\begin{lemma} \label{plocalizationisenough} Suppose $P$ is the set of primes dividing the order of $G$. Then the map $K_0 \mathbb{Z}[P^{-1}] G \rightarrow K_0 \mathbb{Q} G$ is an isomorphism.
\end{lemma}

\subsection{The singular character group {\normalfont{$\text{SC}(G)$}} for finite $G$}
\label{singularcharactergroup}
Write $\text{Conj(G)}$ for the set of conjugacy classes of $G$ and let $Cl(G;k) \defeq \text{Fun}( \text{Conj(G)}, k )$ be the $k$-vector space of class functions of $G$ with values in $k$. It is a standard fact from representation theory that the association $I \mapsto \chi_I$ gives an injection $K_0 kG \hookrightarrow Cl(G;k)$. We call a class function in the image of this inclusion a $k$\emph{-valued virtual character} of $G$.\footnote{Berman's theorem actually shows that the character $\chi_I$ for a $k$-linear representation $I$ is a well-defined function on $k$-\emph{conjugacy classes} of $G$ and the irreducible representations form an orthogonal basis of the space Fun($k$-Conj$(G), k)$ with respect to the scalar product given by $ \left\langle \chi_1, \chi_2 \right\rangle \defeq \frac{1}{|G|} \sum_{g \in G} \chi_1(g) \chi_2(g^{-1}) $. In particular the number $r_k$ is equal to the number of $k$-conjugacy classes of $G$. See \cite{curtis1981methods}, Theorem 21.5.}

The character of the regular representation $kG$ is given by
$$ \chi_{kG}(g) =
\begin{cases}
|G| & \mbox{ if } g = 1 \\
0 & \mbox{ else.} \\
\end{cases} $$
In general, for $I$ any $k$-representation, the value of the character $\chi_I$ at $1$ is $\chi_I (1) = \text{dim}(I)$. It follows that we have a commutative square
$$\xymatrix{
K_0 kG \ar[r] \ar@{^{(}->}[d]^\chi & K_0 k \ar@{^{(}->}[d]^{\text{dim}} \\
Cl(G;k) \ar[r]^{\text{ev}_1} & k,
}$$
where the top arrow is induced by restriction along $k \rightarrow kG$. From this we can deduce that
$$ \widetilde{K_0 }kG = \text{ker}( K_0 kG \rightarrow K_0 k ) \hookrightarrow \text{Fun}( \text{Conj}(G) \backslash \left\lbrace [1] \right\rbrace, k ),$$
in other words, we can interpret the reduced $K$-theory group as the set of $k$-valued virtual characters defined on non-trivial conjugacy classes.

Fix a prime $p$. An element $g$ of $G$ is called singular with respect to $p$ if $p$ divides the order of $g$. Write $\text{Conj}_p(G)$ for the set of $p$-singular conjugacy classes of $G$. The following theorem can be found in Serre \cite{serre1977linear}, Chapter 16, Theorem 34 and 36:
\begin{theorem} \label{padiccharacters}
The map $K_0 \mathbb{Z}_p G \rightarrow K_0 \mathbb{Q}_p G$ is split injective, and the image consists of all virtual representations with characters vanishing on $p$-singular elements of $G$.
\end{theorem}

\begin{definition}
As a consequence, we can identify the cokernel of $K_0 \mathbb{Z}_p G \rightarrow K_0 \mathbb{Q}_p G$ with the set of virtual characters defined on $\text{Conj}_p(G)$. Note that such a character always takes values in $\mathbb{Q}(\zeta_n)$ where $n$ is the order of the group $G$ and $\zeta_n$ is a primitive $n$-th root of unity.  We write $\text{SC}_p(G)$ for the subgroup of Fun($\text{Conj}_p(G), \mathbb{Q}(\zeta_n))$ spanned by those characters and we call them $p$-\emph{singular virtual characters} of $G$. 
\end{definition}

\begin{lemma} \label{singularcharacters}
Let $G$ be finite and $n$ be its order. The group $\text{SC}(G)$ of singular characters of $G$ from Definition \ref{singularcharactersdefinition} is isomorphic to the cokernel of the map
$$ \bigoplus_{p ~\text{prime}, p | n} K_0 \mathbb{Z}_p G \rightarrow \bigoplus_{p ~\text{prime}, p | n}  K_0 \mathbb{Q}_p G.$$
As a consequence of the previous remark, this can be identified with the subgroup of the group of functions $$\text{Fun}\left( \coprod_{p ~\text{prime}, p | n } \text{Conj}_p(G), \mathbb{Q}(\zeta_n) ) \right)$$ consisting of tuples $(\chi_p)$ where each $\chi_p$ is the restriction of a $\mathbb{Q}_p$-valued virtual character of $G$ to the set of $p$-singular elements of $G$. In other words, we have an isomorphism
$$ \text{SC}(G) \cong \bigoplus_{p ~\text{prime}, p | n} \text{SC}_p(G).$$
\end{lemma}

\begin{remark} \label{SCisfree} The group $K_{0} \mathbb{Z}_p G$ is free and of rank $r_{\mathbb{F}_p}$ where $r_{\mathbb{F}_p}$ is the number of irreducible $\mathbb{F}_p$-representations by Serre \cite{serre1977linear}, Chapter 14, Corollary 3 and Chapter 16, Corollary 1. From this it follows that $\text{SC}(G)$ is finitely generated free of rank
$$ r_{\text{SC}} = \sum_{p | ~ |G|}( r_{\mathbb{Q}_p} - r_{\mathbb{F}_p}),$$
since it is isomorphic to the sum of the cokernels of the split injective maps
$$K_0 \mathbb{Z}_p G \rightarrow K_0 \mathbb{Q}_p G $$
between free abelian groups of rank $r_{\mathbb{F}_p}$ and $r_{\mathbb{Q}_p}$, respectively. 
\end{remark}

\begin{remark} \label{scforpgroups}
If $G$ is a finite $p$-group, by the above lemma we have $\text{SC}(G) \cong \text{cok}( K_0 \mathbb{Z}_p G \rightarrow K_0 \mathbb{Q}_p G)$. Moreover, since every non-trivial element of $G$ is $p$-singular, by Theorem \ref{padiccharacters} the image of $K_0 \mathbb{Z}_p G \rightarrow K_0 \mathbb{Q}_p G$ is given by those virtual representations $I$ for which their character $\chi_I$ vanishes away from $1$, which means that $K_0 \mathbb{Z}_p G$ is generated by the free modules. Hence we have an isomorphism $ \text{SC}(G) \cong \widetilde{K_0} \mathbb{Q}_p G.$
\end{remark}

\begin{proof}  Let $P$ be the set of primes dividing the order of $G$. We note that the spectra  $\textbf{Wh}( \mathbb{Q}; G )$ and $\textbf{Wh}( \mathbb{Q}_p; G )$ are connective by \ref{negativektheoryfield}. The groups $K_{-n} \mathbb{Z}_p G$ vanish for $n>0$ (see \cite{carterlocalization}, Page 619), in other words $\textbf{Wh}( \mathbb{Z}_p; G )$ is connective as well. Recall the pullback square
$$ 
\xymatrix{
\textbf{Wh}( \mathbb{Z}; G) \ar[r] \ar[d] & \bigvee_{p \in P} \textbf{Wh}( \mathbb{Z}_p; G ) \ar[d] \\
\textbf{Wh}( \mathbb{Z}[P^{-1}]; G ) \ar[r] & \bigvee_{p \in P} \textbf{Wh}( \mathbb{Q}_p; G )
}$$
of Corollary \ref{whiteheadlocalization}. Denote by $\textbf{C}$ the common vertical cofiber, i.e.
\begin{align*}
\textbf{C} & \defeq \text{cof}\left( \textbf{Wh}( \mathbb{Z}; G) \rightarrow \textbf{Wh}( \mathbb{Z}[P^{-1}]; G )\right) \\ 
& \simeq  \text{cof}\left(  \bigvee_{p \in P} \textbf{Wh}( \mathbb{Z}_p; G ) \rightarrow \bigvee_{p \in P} \textbf{Wh}( \mathbb{Q}_p; G ) \right). 
\end{align*}
Since the spectrum $\textbf{C}$ is a cofiber of a map of connective spectra, it is connective as well and we have
$$ \pi_0 \textbf{C} = \text{cok} \left(\bigoplus_{p ~\text{prime}, p | n} \widetilde{K_0 }\mathbb{Z}_p G \rightarrow \bigoplus_{p ~\text{prime}, p | n}  \widetilde{K_0 }\mathbb{Q}_p G \right).$$
Note that the summand of $K_0 \mathbb{Q}_p G$ corresponding to free $\mathbb{Q}_p G$-modules lies in the image of $K_0 \mathbb{Z}_p G \rightarrow K_0 \mathbb{Q}_p G$ hence the cokernel does not change when going to unreduced $K$-theory, therefore
$$ \pi_0 \textbf{C} \cong \text{cok} \left(\bigoplus_{p ~\text{prime}, p | n} K_0 \mathbb{Z}_p G \rightarrow \bigoplus_{p ~\text{prime}, p | n}  K_0 \mathbb{Q}_p G \right).$$
There is a natural map $\textbf{C} \rightarrow \textbf{SC}(G)$ induced by the map $\textbf{Wh}( \mathbb{Z}[P^{-1}]; G ) \rightarrow \textbf{Wh}( \mathbb{Q}; G )$, which by Lemma \ref{plocalizationisenough} is an isomorphism in $\pi_0$. The $5$-lemma thus implies that $\pi_0 \textbf{C} \cong \pi_0 \textbf{SC}(G) = \text{SC}(G)$.
\end{proof}


\begin{lemma} \label{shortexactsequencefornegativektheory}
Let $G$ be finite. There is a natural short exact sequence
$$ 0 \rightarrow \widetilde{ K_0 }\mathbb{Q} G\rightarrow \text{SC}(G) \rightarrow K_{-1} \mathbb{Z} G \rightarrow 0, $$
which is a free resolution of the abelian group $K_{-1} \mathbb{Z} G$, and the map $\widetilde{ K_0 }\mathbb{Q} G\rightarrow \text{SC}(G)$ simply sends a rational representation $I$ to the corresponding singular character $(\chi_p)_{p ~ \text{prime}}$ defined as $\chi_p(g) \defeq \chi_I(g)$, where $\chi_I$ is the character of $I$.
\end{lemma}

\begin{proof}
The long exact sequence of the fiber sequence
$$ \textbf{Wh}( \mathbb{Z}; G ) \rightarrow \textbf{Wh}( \mathbb{Q}; G ) \rightarrow \textbf{SC}(G) $$
gives the exact sequence
$$ \widetilde{ K_0 }\mathbb{Z} G  \rightarrow \widetilde{ K_0 }\mathbb{Q} G\rightarrow \text{SC}(G) \rightarrow K_{-1} \mathbb{Z} G \rightarrow 0, $$
since $\textbf{Wh}( \mathbb{Q}; G ) $ is connective by \ref{negativektheoryfield}. Theorem \ref{swansprojectivetheorem} implies that this sequence splits off to the left, giving the claimed short exact sequence. The group $\widetilde{ K_0 }\mathbb{Q} G$ is free since it is isomorphic to \mbox{$\text{ker}( K_0 \mathbb{Q} G \rightarrow K_0 \mathbb{Q})$,} which is, as a subgroup of the free abelian group $K_0 \mathbb{Q} G$, again free and $\text{SC}(G)$ is free as discussed in Remark \ref{SCisfree}.
Lastly, the claim that $\widetilde{ K_0 }\mathbb{Q} G\rightarrow \text{SC}(G)$ sends a representation to the singular character $(\chi_p)_{p ~ \text{prime}}$ follows from the fact that $\widetilde{ K_0 }\mathbb{Q} G\rightarrow \text{SC}(G)$ factors as
$$ \widetilde{ K_0 }\mathbb{Q} G\rightarrow \bigoplus_{p | ~ |G|}   K_0 \mathbb{Q}_p G \rightarrow  \text{SC}(G),$$
where the first map is induced by the ring homorphisms $\mathbb{Q} \rightarrow \mathbb{Q}_p$ and the second by Lemma \ref{singularcharacters}.
\end{proof}

Define the \emph{Bockstein morphism} $\beta_n  \colon \textbf{H} \mathbb{Z}/n \rightarrow \Sigma \textbf{H} \mathbb{Z}$ as the boundary morphism to the fiber sequence of spectra
$$ \textbf{H}\mathbb{Z} \xrightarrow{n \cdot} \textbf{H}\mathbb{Z} \rightarrow \textbf{H} \mathbb{Z}/n. $$ The following theorem is now a consequence of Lemma \ref{shortexactsequencefornegativektheory} and Theorem \ref{negativektheory}.
\begin{theorem}
Let $G$ be finite and let $s$ be the number of irreducible $\mathbb{Q}$-representations with even Schur index but odd local Schur index at every prime $p$ dividing the order of $G$. The map of spectra $\textup{\textbf{Wh}}(\mathbb{Z}; G)[-1,0] \rightarrow \textup{\textbf{Wh}}(\mathbb{Q}; G)[-1,0]$ factorizes as
$$\textup{\textbf{Wh}}(\mathbb{Z}; G)[-1,0] \xrightarrow{p} \Sigma^{-1} \textup{\textbf{H}} (\mathbb{Z}/2)^s \xrightarrow{{(\beta_2)^s}}  \textup{\textbf{H}} \mathbb{Z}^s \xrightarrow{i} \textup{\textbf{H}} \widetilde{K_0 }\mathbb{Q} G \cong \textup{\textbf{Wh}}(\mathbb{Q}; G)[-1,0]$$
where the map $p$ is given by the Postnikov truncation of $\textup{\textbf{Wh}}(\mathbb{Z}; G)[-1,0]$ followed by the projection onto the torsion summand of $K_{-1} \mathbb{Z} G$ and the map $i  \colon \textup{\textbf{H}} \mathbb{Z}^s \rightarrow \textup{\textbf{H}} \widetilde{K_0 }\mathbb{Q} G$ is induced by the inclusion of all linear combinations of the irreducible $\mathbb{Q}$-representations that contribute to $s$.
\end{theorem}

We want to stress the importance of this theorem to the reader in regard to the analysis of the map $\widetilde{ K_0 }\mathbb{Z} G  \rightarrow \widetilde{ K_0 }\mathbb{Q} G$. The first major obstruction for generalizing the triviality of $\widetilde{ K_0 }\mathbb{Z} G  \rightarrow \widetilde{ K_0 }\mathbb{Q} G$ from finite to arbitrary groups lies in the fact that while $\textbf{Wh}(\mathbb{Z}; G)[-1,0] \rightarrow \textbf{Wh}(\mathbb{Q}; G)[-1,0]$ for finite groups $G$ is trivial on homotopy groups, it is not the trivial map of spectra, unless $s = 0$.

\begin{proof}
The map $\widetilde{ K_0 }\mathbb{Z} G  \rightarrow \widetilde{ K_0 }\mathbb{Q} G$ is zero for $G$ being finite by Theorem \ref{swansprojectivetheorem} and $\textbf{Wh}(\mathbb{Q}; G)[-1,0]$ is actually concentrated in degree $0$ since the negative $K$-theory of $\mathbb{Q}G$ vanishes. This means we can apply Lemma \ref{usefullemma} to see that 
$$\textbf{Wh}(\mathbb{Z}; G)[-1,0] \rightarrow \textbf{Wh}(\mathbb{Q}; G)[-1,0]$$
factors through a unique map
$$ \Sigma^{-1} \textbf{H} K_{-1} \mathbb{Z} G \rightarrow \textbf{H} \widetilde{K_0 }\mathbb{Q}G. $$
It corresponds under Lemma \ref{mapsofdegreeoneandexactsequence} to the short exact sequence
$$ 0 \rightarrow \widetilde{ K_0 }\mathbb{Q} G\rightarrow \text{SC}(G) \rightarrow K_{-1} \mathbb{Z} G \rightarrow 0 $$
from Lemma \ref{shortexactsequencefornegativektheory}.
Now $K_{-1} \mathbb{Z} G = \mathbb{Z}^r \oplus (\mathbb{Z}/2)^s$ by Theorem \ref{negativektheory}. The abelian group $\widetilde{ K_0 }\mathbb{Q} G$ is f.g.\ free and maps of degree 1 of the form $\Sigma^{-1} \textbf{H} \mathbb{Z} \rightarrow \textbf{H} \mathbb{Z}$  are necessarily zero by Lemma \ref{mapsofdegreeoneandexactsequence}, since $\text{Ext}^1_\mathbb{Z}(\mathbb{Z}, \mathbb{Z}) = 0$ as $\mathbb{Z}$ is free. This means the map $ \Sigma^{-1} \textbf{H} K_{-1} \mathbb{Z} G \rightarrow \textbf{H} \widetilde{K_0 }\mathbb{Q}G $ further factors through the $2$-torsion, i.e. as 
$$\Sigma^{-1} \textbf{H} (\mathbb{Z}/2)^s \rightarrow \textbf{H} \widetilde{K_0 }\mathbb{Q}G. $$
The generators of $\widetilde{K_0 }\mathbb{Q}G$ are given by the isomorphism classes of non-trivial irreducible $\mathbb{Q}$-representations of $G$. By Theorem \ref{negativektheory}, each of these contributes to a single $\mathbb{Z}/2$-summand in $K_{-1} \mathbb{Z}G$ iff it has even global Schur index but odd local Schur index at every prime $p$ dividing the order of $G$, giving rise to a Bockstein morphism $\beta_2$. This means that we have $s$ many linearly independent singular characters $\frac{\chi_I}{2}$, one for each such irreducible representation $I$, which together span a copy of $\mathbb{Z}^s$ in $\text{SC}(G)$. Therefore we have the sub-exact sequence
$$ \mathbb{Z}^s \xrightarrow{2\cdot} \mathbb{Z}^s \rightarrow (\mathbb{Z}/2)^s $$
of the short exact sequence we started with. In other words, the map $\Sigma^{-1} \textbf{H} (\mathbb{Z}/2)^s \rightarrow \textbf{H} \widetilde{K_0 }\mathbb{Q}G$ factors further through $i$,
$$\Sigma^{-1} \textbf{H} (\mathbb{Z}/2)^s \xrightarrow{(\beta_2)^s} \textbf{H} \mathbb{Z}^s \xrightarrow{i} \textbf{H} \widetilde{K_0 }\mathbb{Q}G,$$
with $i$ being the inclusion of the subgroup of $\widetilde{K_0 }\mathbb{Q}G$ generated by all the irreducible $\mathbb{Q}$-representations that contribute to $s$.
\end{proof}

\section{The map $K_0 \mathbb{Z} G \rightarrow K_0 \mathbb{Q} G$ for infinite groups} \label{infinitegroups}

%
%
%

The following section is concerned with proving the main theorem. Throughout, assume that $G$ satisfies the Farrell-Jones conjecture and that $E \text{Fin}$ is a fixed model for the classifying space of finite subgroups together with a chosen CW-structure $(E \text{Fin}^{(k)})_{k \in \mathbb{N}}$. Write
$$ (f,g)  \colon \coprod_{i \in I} G/H_i \times S^0 \rightarrow \coprod_{j \in J} G/K_j $$
for the degree $0$ attaching map of $E\text{Fin}$ with the $H_i$ and $K_j$ being finite subgroups of $G$.
For a functor $F  \colon \text{Or}G \rightarrow \text{Ab}$ define
$$ \text{ker}^F \defeq \text{ker}( F(f)-F(g) )  \colon \bigoplus_{i \in I} F(G/H_i) \rightarrow \bigoplus_{j \in J} F(G/K_j). $$

\begin{theorem} \label{technicallemma}
There is an exact sequence
$$ 0 \rightarrow \text{ker}^{\widetilde{K_0 }\mathbb Q}  \rightarrow \text{ker}^{\mathrm{SC}} \rightarrow \text{ker}^{K_{-1} \mathbb{Z}} \rightarrow \text{im}( \widetilde{K_0 }\mathbb{Z}G \rightarrow \widetilde{K_0 }\mathbb{Q}G ) \rightarrow 0 $$
and the map $\text{ker}^{K_{-1} \mathbb{Z}} \rightarrow \text{im}( \widetilde{K_0 }\mathbb{Z}G \rightarrow \widetilde{K_0 }\mathbb{Q}G )$ is the connecting map induced from the snake lemma applied to the diagram
$$\xymatrix{
0 \ar[r]  & \bigoplus_{i \in I} \widetilde{K_0 }\mathbb Q(H_i) \ar[r]\ar[d] & \bigoplus_{i \in I} \mathrm{SC}(H_i) \ar[r] \ar[d] & \bigoplus_{i \in I} K_{-1} \mathbb{Z}(H_i)  \ar[r] \ar[d]^{f-g} & 0 \\
0 \ar[r] & \bigoplus_{j \in J} \widetilde{K_0 }\mathbb Q(K_j) \ar[r] & \bigoplus_{j \in J} \mathrm{SC}(K_j) \ar[r] & \bigoplus_{j \in J} K_{-1} \mathbb{Z}(K_j)  \ar[r] & 0.
}$$

\end{theorem}

Before we begin with the proof, we need a few more arguments. If $D$ is a $1$-category, then the category of functors $D \rightarrow \text{Ab} \subset \text{Sp}$ with values in the heart $\text{Ab}$ of $\text{Sp}$ is again a $1$-category and thus a natural transformation $\eta  \colon \textbf{A} \implies \textbf{B}$  between two functors $\textbf{A}, \textbf{B}  \colon D \rightarrow \text{Sp}$ with values in the heart is the zero map in the category $\text{Fun}(D , \text{Sp})$ if and only if its value on all the components $\eta_d  \colon \textbf{A}(d) \rightarrow \textbf{B}(d)$ is the zero homomorphism for all $d \in D$. Note that here it is essential that the category of functors with values in the heart of $\text{Sp}$ is again a $1$-category, it is not true in general that a natural transformation between two functors with values in spectra is zero if all its components are zero maps.

Now, since the map $\widetilde{K_0 }\mathbb{Z} H \rightarrow \widetilde{K_0 }\mathbb{Q} H $ vanishes for all finite subgroups $H$, the natural transformation $\textbf{H} \widetilde{K_0 }\mathbb{Z} - \implies \textbf{H} \widetilde{K_0 }\mathbb{Q} -$ becomes the zero map when restricted to the subcategory $ \text{Or}G_{\text{Fin}}$. Furthermore, $\textbf{Wh}(\mathbb Q; -)[-1,0]$ is as a functor on $\text{Or}G_{\text{Fin}}$ concentrated in degree $0$ since the negative $K$-theory of the group algebras $\mathbb{Q}H$ vanishes for $H$ being finite. By using the object-wise $t$-structure on $\text{Fun}(\text{Or}G_\text{Fin} , \text{Sp})$ (see Definition \ref{objectwisetstructure}), we are now in a position to apply Lemma \ref{usefullemma} with $\mathcal{C} = \text{Fun}(\text{Or}G_\text{Fin} , \text{Sp})$, and the map $f$ in question being the natural transformation $\textbf{Wh}(\mathbb Z; -)[-1,0] \implies \textbf{Wh}(\mathbb Q; -)[-1,0]$. Lemma \ref{usefullemma} states that the natural transformation of functors
$$\textbf{Wh}(\mathbb Z; -)[-1,0] \implies \textbf{Wh}(\mathbb Q; -)[-1,0]$$
descends to a unique natural transformation of functors $\text{Or}G_{\text{Fin}} \rightarrow \text{Sp} $,
$$ \Sigma^{-1} \textbf{H} K_{-1} \mathbb{Z} - \implies  \textbf{H} \widetilde{K_0 }\mathbb{Q} -. $$

\begin{proposition} \label{pi1matters}
If $G$ satisfies the Farrell-Jones conjecture, the image of the map $\widetilde{K_0 }\mathbb{Z} G \rightarrow \widetilde{K_0 }\mathbb{Q} G$ agrees with the image of
$$ \pi_1 E \text{Fin} \otimes_{\text{Or}G} \textbf{H} K_{-1} \mathbb{Z} \rightarrow \widetilde{K_0 }\mathbb{Q} G,$$
as well as with the image of the map
$$ \pi_1 E \text{Fin}^{(1)} \otimes_{\text{Or}G} \textbf{H} K_{-1} \mathbb{Z} \rightarrow \widetilde{K_0 }\mathbb{Q} G$$
induced by the inclusion $E \text{Fin}^{(1)} \subset E \text{Fin}$.
\end{proposition}

\begin{proof}
The fiber sequence
$$ \textbf{H} \widetilde{K_0 }\mathbb{Z} - \implies \textbf{Wh}(\mathbb Z; G)[-1,0] \implies  \Sigma^{-1} \textbf{H} K_{-1} \mathbb{Z} -$$
of functors leads to the exact sequence \\
\begin{tikzcd}[row sep=1.5em, column sep = 0.5em]
\cdots \arrow[r] & \pi_0 E \text{Fin} \otimes_{\text{Or}G} ( \textbf{Wh}(\mathbb Z; -)[-1,0] ) \arrow[rr] &[-5.5em] \arrow[d, phantom, ""{coordinate, name=Z}] &[-5.5em] \pi_0 E \text{Fin} \otimes_{\text{Or}G} \Sigma^{-1} \textbf{H} K_{-1} \mathbb{Z} - \arrow[dl,rounded corners=8pt,curvarr=Z] \\
  & & \pi_{-1} E \text{Fin} \otimes_{\text{Or}G} \textbf{H} \widetilde{K_0 }\mathbb{Z} - \arrow[r] & \cdots.
\end{tikzcd} \\

Since $\textbf{H} \widetilde{K_0 }\mathbb{Z} -$ is a connective functor and smashing with a $G$-space preserves connectivity, the group $\pi_{-1} E \text{Fin} \otimes_{\text{Or}G} \textbf{H} \widetilde{K_0 }\mathbb{Z} -$ vanishes, which means that the map 
$$ \pi_0 E \text{Fin} \otimes_{\text{Or}G} ( \textbf{Wh}(\mathbb Z; -)[-1,0] ) \rightarrow \pi_0 E \text{Fin} \otimes_{\text{Or}G} \Sigma^{-1} \textbf{H} K_{-1} \mathbb{Z} -$$
is an epimorphism.

As discussed before, we have a commuting triangle of natural transformations of functors $\text{Or}G_{\text{Fin}} \rightarrow \text{Sp} $,
$$\xymatrix{
\textbf{Wh}(\mathbb Z; -)[-1,0] \ar@{=>}[d] \ar@{=>}[r] & \textbf{H} \widetilde{K_0 }\mathbb{Q} - \\
\Sigma^{-1} \textbf{H} K_{-1} \mathbb{Z} \ar@{=>}[ru] &
}$$
Taking orbit tensor products with $E \text{Fin}$ and using that $\widetilde{K_0 }\mathbb{Q} - $ satisfies finite assembly, see Remark \ref{rationalktheoryfiniteassembly}, we get the triangle
$$\xymatrix{
\pi_0 E \text{Fin} \otimes_{\text{Or}G} (\textbf{Wh}(\mathbb Z; -)[-1,0]) \ar@{->>}[d] \ar[r] & \widetilde{K_0 }\mathbb{Q} G \\
\pi_1 E \text{Fin} \otimes_{\text{Or}G} \textbf{H} K_{-1} \mathbb{Z}, \ar[ru] &
}$$
which together with Corollary \ref{plocallemma} proves the first statement.

For the second statement we use Lemma \ref{homologypreservesconnectivity} to get that 
$$\pi_1 E \text{Fin}^{(1)} \otimes_{\text{Or}G} \textbf{H} K_{-1} \mathbb{Z}  \rightarrow \pi_1 E \text{Fin} \otimes_{\text{Or}G} \textbf{H} K_{-1} \mathbb{Z} $$
is an epimorphism. This allows us to reduce further to the image of the composition
\begin{align*}
\pi_1 E \text{Fin}^{(1)} \otimes_{\text{Or}G} \textbf{H} K_{-1} \mathbb{Z} \twoheadrightarrow  \pi_1 E \text{Fin} \otimes_{\text{Or}G} \textbf{H} K_{-1} \mathbb{Z} \rightarrow & \pi_0 E \text{Fin} \otimes_{\text{Or}G} \textbf{H} \widetilde{K_0 }\mathbb{Q} \\ 
& \cong \widetilde{K_0 }\mathbb{Q} G.
\end{align*}
\end{proof}

\begin{proof}[Proof of Theorem \ref{technicallemma}]
We can already reduce the image of
$$ \widetilde{K_0 }\mathbb{Z} G \rightarrow \widetilde{K_0 }\mathbb{Q} G$$
to that of the map
$$ \pi_1{E \text{Fin}^{(1)} \otimes_{\text{Or}G} \textbf{H} K_{-1} \mathbb{Z} } \rightarrow \widetilde{K_0 }\mathbb{Q} G $$
thanks to Corollary \ref{pi1matters}.

Applying Lemma \ref{1dimensionalclassifyingspace} to the natural transformation
$$\textbf{H} K_{-1} \mathbb{Z} - \implies  \Sigma \textbf{H} \widetilde{K_0 }\mathbb{Q} -,$$ we get the following commutative diagram of spectra
\[ \tag{$\ast$} \begin{split} \xymatrix{
\textbf{H} ( \bigoplus_{ i \in I }  K_{-1} \mathbb{Z} H_i ) \ar[r] \ar[d]^{f-g} & \Sigma \textbf{H} ( \bigoplus_{ i \in I }  \widetilde{K_0 }\mathbb{Q} H_i ) \ar[d] \\
\textbf{H} ( \bigoplus_{ j \in J }  K_{-1} \mathbb{Z} K_j ) \ar[r]  \ar[d] & \Sigma \textbf{H} ( \bigoplus_{ j \in J }  \widetilde{K_0 }\mathbb{Q} K_j) \ar[d] \\
E \text{Fin}^{(1)} \otimes_{\text{Or}G} \textbf{H} K_{-1} \mathbb{Z} \ar[r] & \Sigma  E \text{Fin}^{(1)} \otimes_{\text{Or}G} \textbf{H} \widetilde{K_0 }\mathbb{Q}
}  \end{split} \] 
with the columns being fiber sequences.

By Lemma \ref{tsnakelemma}, the map induced on $\pi_1$ on the cofibers is equivalent to the map induced by the snake lemma of the diagram
$$\xymatrix{
0 \ar[r]  & \bigoplus_{i \in I} \widetilde{ K_0 }\mathbb{Q}H_i \ar[r] \ar[d] &  \bigoplus_{i \in I} \text{SC}(H_i) \ar[r] \ar[d] & \bigoplus_{i \in I} K_{-1} \mathbb{Z}H_i \ar[r] \ar[d] &  0  \\
0 \ar[r] & \bigoplus_{i \in I} \widetilde{ K_0 }\mathbb{Q}K_j \ar[r] &  \bigoplus_{i \in I} \text{SC}(K_j) \ar[r] &  \bigoplus_{i \in I} K_{-1} \mathbb{Z}K_j \ar[r] & 0
}$$
with exact rows.

This means we get the claimed exact sequence
$$ 0 \rightarrow \text{ker}^{\widetilde{K_0 }\mathbb Q}  \rightarrow \text{ker}^{\text{SC}} \rightarrow \text{ker}^{K_{-1} \mathbb{Z}} \rightarrow \text{im}( \widetilde{K_0 }\mathbb{Z}G \rightarrow \widetilde{K_0 }\mathbb{Q}G ) \rightarrow 0. $$
%
\end{proof}


\section{Virtually cyclic groups}

A group $G$ is called virtually cyclic if it contains a cyclic subgroup of finite index. Virtually cyclic groups can be classified into three families of groups.

\begin{lemma}[See \cite{hempel}, Lemma 11.4.] A group $G$ is virtually cyclic if it is of one of the three forms
\begin{itemize}
\item $G$ is finite.
\item $G$ is finite-by-infinite cyclic. This means that there is an exact sequence of groups
$$ 1 \rightarrow H \rightarrow G \rightarrow C_{\infty} \rightarrow 1 $$
with $H$ being finite, and $C_{\infty}$ an infinite cyclic group. We will call $G$ of type $\text{VC}1$.
\item $G$ is finite-by-infinite dihedral. This means that there is an exact sequence of groups
$$ 1 \rightarrow H \rightarrow G \rightarrow D_{\infty} \rightarrow 1 $$
with $H$ being finite, and $D_{\infty}$ an infinite dihedral group. We will call $G$ of type $\text{VC}2$.
\end{itemize}
\end{lemma}

Before we begin with the analysis of type $1$ or type $2$ virtually cyclic groups, we state some properties of their negative $K$-theory.

\subsection{Negative $K$-theory of virtually cyclic groups}

The following theorem due to Farrell, Jones extends Carter's results to virtually cyclic groups.
\begin{theorem}[\cite{Farrell_Jones_lower_1995}, Theorem 2.1.] \label{negativektheoryvirtuallycyclic}
Let $G$ be a virtually infinite cyclic group. Then
\begin{enumerate}[label=\textup{(\alph*)}]
\item $K_n \mathbb Z G = 0$ for all integers $n \leq -2$. 
\item $K_{-1} \mathbb{Z} G$ is generated by the images of $K_{-1} \mathbb{Z} F$ under the maps induced by
the inclusions $F \subset G$ where $F$ varies over representatives of the conjugacy
classes of finite subgroups of $G$.
\item $K_{-1} \mathbb{Z} G$ is a finitely generated abelian group. 
\end{enumerate}
\end{theorem}

This has a few implications for groups that satisfy the Farrell Jones conjecture.

\begin{corollary} \label{negativektheoryfiniteassembly}
Let $G$ be a group satisfying the Farrell Jones conjecture. Then
\begin{itemize}
\item $K_n \mathbb Z G = 0$ for all integers $n \leq -2$. 
\item The functor $K_{-1} \mathbb{Z} -$ satisfies finite assembly in the sense that 
$$ K_{-1} \mathbb{Z} G \cong \text{colim}_{G/H \in \text{Or}G_\text{Fin}} K_{-1} \mathbb{Z} H.$$
\item Let $k$ be a field of characteristic $0$. Then $K_n k G = 0$ for all integers $n \leq -1$.
\end{itemize}
\end{corollary}

\begin{proof}
Theorem \ref{negativektheoryvirtuallycyclic} (a) implies that the functor $ \textbf{K} \mathbb{Z} -$ is $(-1)$-connective, when restricted to the category $\text{Or}G_\text{VCyc}$ and thus
$$\textbf{K} \mathbb{Z} G \simeq \text{colim}_{G/H \in \text{Or}G_\text{VCyc}} \textbf{K} \mathbb{Z} H$$
is $(-1)$-connective as well, since the subcategory of $(-1)$-connective spectra is closed under colimits. This implies the first statement. Furthermore, since the functor $\pi_{-1}$ is a left adjoint when restricted to $(-1)$-connective spectra, we get the isomorphism
$$ K_{-1} \mathbb{Z} G \cong \text{colim}_{G/H \in \text{Or}G_\text{VCyc}} K_{-1} \mathbb{Z} H.$$
Hence to show the second statement we need to show that
$$ \text{colim}_{G/H \in \text{Or}G_\text{Fin}} K_{-1} \mathbb{Z}H \rightarrow \text{colim}_{G/H \in \text{Or}G_\text{VCyc}} K_{-1} \mathbb{Z}H $$
induced by the inclusion $\text{Or}G_\text{Fin} \subset \text{Or}G_\text{VCyc}$ is an isomorphism. Theorem \ref{fintovcycissplit} already states that it is injective. Surjectivity is implied by Theorem \ref{negativektheoryvirtuallycyclic} (b).
The statement that $K_n k G = 0$ for $k$ a field of characteristic $0$ and $n  \leq -1$ follows similarly from the observation that $ K_n  k -$ satisfies finite assembly, see \ref{rationalktheoryfiniteassembly} and that the negative $K$-theory of $k H$ vanishes for finite groups $H$, see \ref{negativektheoryfield}.
\end{proof}

\subsection{Virtually cyclic groups of type 1}

In the following fix a group $G$ of type VC1 and write $H$ for the unique maximal finite subgroup. Write $\pi  \colon G \rightarrow G/H \cong C_\infty$ for the canonical projection. Since the kernel $H$ of $G \rightarrow C_\infty$ is finite, the following is easy to show:

\begin{lemma}
Let $G$ be of type VC1. Then a model of the classifying space $E(G;\text{Fin})$ is given by $\mathbb{R}$ with the action lifted from the translation action of $C_\infty = G/H$.
\end{lemma}

A $C_\infty$-CW-structure of $\mathbb{R}$ with the translation action can be described with the following pushout square.
$$\begin{tikzcd}
C_\infty / 1 \times S^0 \arrow[r, "{(id , t \cdot )}"] \arrow[d] & C_\infty / 1 \arrow[d] \\
C_\infty / 1 \times D^1 \arrow[r] & \mathbb{R}
\end{tikzcd}$$
This generalizes for $G$ being of type VC1 in the following way. Let $\tilde{t} \in G$ be a choice of lift of the generator $t$ in $C_\infty$. Then the following is a pushout square of $G$-spaces.
$$\begin{tikzcd}
G / H \times S^0 \arrow[r, "{(id, \tilde{t}\cdot)}"] \arrow[d] & G / H \arrow[d] \\
G / H \times D^1 \arrow[r] & \mathbb{R}
\end{tikzcd} \label{pushoutsquareforVC1}$$
Applying Lemma \ref{1dimensionalclassifyingspace} now states that if $F$ is any functor $ \text{Or}G \rightarrow \text{Sp}$,  then there is a fiber sequence
$$F(H) \xrightarrow{1-\tilde{t}} F(H) \rightarrow E(G;\text{Fin}) \otimes F.$$

The functors $K_0 \mathbb{Q}-$ and $K_{-1} \mathbb{Z}-$ satisfy finite assembly (see Remark \ref{rationalktheoryfiniteassembly} as well as Corollary \ref{negativektheoryfiniteassembly}). We thus have the exact sequences
$$ K_0 \mathbb{Q} H \xrightarrow{ 1 - t } K_0 \mathbb{Q} H \rightarrow K_0 \mathbb{Q} G \rightarrow 0$$
and 
$$ K_{-1} \mathbb{Z} H \xrightarrow{ 1 - t } K_{-1} \mathbb{Z} H \rightarrow K_{-1} \mathbb{Z} G \rightarrow 0.$$

\begin{theorem} \label{ktheoryoverrationalsofvirtuallycyclicisfree}
Let $G$ be a group of type $\text{VC}1$. Then $K_0 \mathbb{Q} G$ is a finitely generated and free abelian group.
\end{theorem}

\begin{proof}[Proof of Theorem \ref{ktheoryoverrationalsofvirtuallycyclicisfree}]
Since $K \mathbb{Q} -$ satisfies finite assembly, as remarked above, we have the exact sequence
$$ K_0 \mathbb{Q} H \xrightarrow{ 1 - t } K_0 \mathbb{Q} H \rightarrow K_0 \mathbb{Q} G \rightarrow 0.$$

What is left to understand is the action of $t$ on $K_0 \mathbb{Q} H$. The endomorphisms of the object $G/H$ in the category $\text{Or}G$ are equal to $N(H)/H = G/H = C_\infty = \left\langle t \right\rangle$. By Theorem \ref{ktheoryfunctor} this action of $t$ on $K_0 \mathbb{Q} H$ sends a representation $V = (V,\rho)$ to the representation $V_t \defeq (V, \rho(\tilde{t} (-) \tilde{t}^{-1}) ) $. Let $\text{Irr}_\mathbb{Q}(H)$ be the set of isomorphism classes of irreducible representations of $H$ over $\mathbb{Q}$. If $V$ is irreducible, then so is $V_t$, hence
$$K_0 \mathbb{Q} G \cong \text{cok}( 1 - t ) \cong (K _0 \mathbb{Q} H)_{C_\infty} = \mathbb{Z}[ \text{Irr}_\mathbb{Q}(H) ]_{C_\infty} \cong \mathbb{Z}[ \text{Irr}_\mathbb{Q}(H) /  \equiv  ]$$
with $\equiv$ being the equivalence relation generated by $V \equiv V_t$. Hence $K_0 \mathbb{Q} G$, is free generated by the finite set of $C_\infty$-equivalence classes of rational irreducible representations of $H$.
\end{proof}

\begin{corollary} \label{rationalizationtrivialforvirtuallycyclic}
The map $\widetilde{K_0 }\mathbb{Z} G \rightarrow \widetilde{K_0 }\mathbb{Q} G$ is trivial for $G$ of type VC1. 
\end{corollary}

\begin{proof} A virtually cyclic group trivially satisfies the Farrell-Jones conjecture. Theorem \ref{FJimpliesrationalvanishing} implies that the image of the map $\widetilde{K_0 }\mathbb{Z} G \rightarrow \widetilde{K_0 }\mathbb{Q} G$ is torsion, which has to be trivial, since Theorem \ref{ktheoryoverrationalsofvirtuallycyclicisfree} states that $\widetilde{K_0 }\mathbb{Q} G$ is free.
\end{proof}
\subsection{Virtually cyclic groups of type 2} \label{virtuallycyclictype2}

In the following fix a group $G$ of type VC2. Write $\pi  \colon G \rightarrow G/H \cong D_\infty$ for the canonical projection. Let $H$ be the kernel of $\pi$. It is not difficult to show that a group $G$ is of type $\text{VC}2$ iff $G \cong K_1 *_H K_2$, where $K_1$ and $K_2$ are two finite groups that both contain $H$ as an index $2$ subgroup (see e.g. \cite{Lima_Goncalves_2013}, Theorem 17). Waldhausen \cite{waldhausen_free_products} showed that in this case there is a fiber sequence
$$ \textbf{K} RH \rightarrow \textbf{K} RK_1 \vee \textbf{K} RK_2 \rightarrow \textbf{K} RG/\textbf{Nil}^W_R, $$
where the spectrum $\textbf{K} RG/\textbf{Nil}^W_R$ is a natural split summand of $\textbf{K} RG$, i.e. we have
$$ \textbf{K} RG \simeq \textbf{K} RG/\textbf{Nil}^W_R \vee \textbf{Nil}^W_R. $$
Moreover, the spectrum $\textbf{Nil}^W_R$ is contractible if $RH$ is a regular coherent ring. We will get the same result using a geometric understanding of the classifying space $E(G;\text{Fin})$.

We can equip $\mathbb{R}$ with an action of $D_\infty = \left\langle a, b | a^2 = b^2 = 1 \right\rangle$ by sending $a$ to the reflection around $0$ and $b$ to the reflection around $1/2$. The $D_\infty$-space $\mathbb{R}$ is easily seen to be a model for $E(D_\infty;\text{Fin})$. If $G$ is any group of type VC2, we can equip $\mathbb{R}$ with a $G$-action via the projection $G \rightarrow D_\infty$. Since the kernel $H$ of $G \rightarrow D_\infty$ is finite, we conclude the following:

\begin{lemma} \label{realsareefin}
Let $G$ be virtually cyclic of type 2. The $G$-space $\mathbb{R}$ with action lifted from the projection $G \rightarrow D_\infty$ is a model for $E(G;\text{Fin})$.
\end{lemma}

Consequently, we get a nice pushout description for the $G$-space $E(G;\text{Fin})$.

\begin{lemma} \label{vc2pushout}
Suppose $G = K_1 *_H K_2$ is of type $\text{VC}2$. Then there is a pushout square
$$\xymatrix{
G/H \times S^0 \ar[r] \ar[d]	& G/K_1 \sqcup G/K_2 \ar[d] \\
G/H \times D^1 \ar[r]		& E\text{Fin}
}$$
of $G$-spaces, giving $E\text{Fin}$ a 1-dimensional $G$-CW-structure.
\end{lemma}

\begin{proof} By Lemma \ref{realsareefin} the space $\mathbb{R}$ with the action lifted from the projection $\pi  \colon G \rightarrow D_\infty$ is a model for $E\text{Fin}$. This means we may as well assume that $G = D_\infty = \left\langle a, b | a^2 = b^2 = 1\right\rangle$, i.e. $H = \set{1}$, $K_1 = \left\langle a \right\rangle$, $K_2 = \left\langle b \right\rangle $. Now it is an elementary exercise to see that $\mathbb{R}$ indeed fits into a pushout square of the shape
$$\xymatrix{
D_\infty \times S^1 \ar[r] \ar[d]	& D_\infty/ \left\langle a \right\rangle \sqcup D_\infty/ \left\langle b \right\rangle \ar[d] \\
D_\infty \times D^1 \ar[r]		& \mathbb{R}.
}$$
\end{proof}

Write $\iota_i$ for the inclusions $H \hookrightarrow K_i$. The functors $K_0 \mathbb{Q}-$ and $K_{-1} \mathbb{Z}-$ satisfy finite assembly (see Remark \ref{rationalktheoryfiniteassembly} as well as Corollary \ref{negativektheoryfiniteassembly}) therefore as a consequence of the pushout square from Lemma \ref{vc2pushout} we get the exact sequences
\begin{align*}
 \widetilde{K_0 }\mathbb{Q} H \xrightarrow{(\iota_1,-\iota_2)} \widetilde{K_0 }\mathbb{Q} K_1 \oplus \widetilde{K_0 }\mathbb{Q} K_2 \rightarrow \widetilde{K_0 \mathbb{Q} G } \rightarrow 0  \\
  K_{-1} \mathbb{Z} H \xrightarrow{(\iota_1,-\iota_2)} K_{-1} \mathbb{Z} K_1 \oplus K_{-1} \mathbb{Z} K_2 \rightarrow K_{-1} \mathbb{Z} G \rightarrow 0
\end{align*}
and we have a long exact sequence from Lemma \ref{technicallemma}
$$ 0 \rightarrow \text{ker}^{\widetilde{K_0 }\mathbb Q}  \rightarrow \text{ker}^{\text{SC}} \rightarrow \text{ker}^{K_{-1} \mathbb{Z}} \rightarrow \text{im}( \widetilde{K_0 }\mathbb{Z}G \rightarrow \widetilde{K_0 }\mathbb{Q}G ) \rightarrow 0 $$
with 
\begin{align*}
\text{ker}^{\widetilde{K_0 }\mathbb Q} \cong \text{ker} \left( \widetilde{K_0 }\mathbb{Q} H \xrightarrow{(\iota_1,-\iota_2)} \widetilde{K_0 }\mathbb{Q} K_1 \oplus \widetilde{K_0 }\mathbb{Q} K_2 \right) \\
\text{ker}^{\text{SC}} \cong \text{ker} \left( \text{SC}(H) \xrightarrow{(\iota_1,-\iota_2)} \text{SC}(K_1) \oplus \text{SC}(K_2) \right) \\
\text{ker}^{K_{-1} \mathbb{Z}} \cong \text{ker} \left( K_{-1} \mathbb{Z} H \xrightarrow{(\iota_1,-\iota_2)} K_{-1} \mathbb{Z} K_1 \oplus K_{-1} \mathbb{Z} K_2 \right).
\end{align*}
%
%

\begin{remark}
We will construct an example of a group $G$ of type $\text{VC}2$, for which the map $ \widetilde{K_0 }\mathbb{Z} G \rightarrow \widetilde{K_0 }\mathbb{Q} G $ is non-trivial, in section \ref{counterexample}.
\end{remark}

\section{A counterexample to the integral $\widetilde{K_0 }\mathbb{Z}G$-to-$\widetilde{K_0 }\mathbb{Q}G$ conjecture} \label{counterexample}

The following section is concerned with an example of a group $G$ with the property that $ \widetilde{ K_0 }\mathbb{Z} G \rightarrow \widetilde{ K_0 }\mathbb{Q} G$ is non-trivial.

For the construction take $Q_{16}$ contained in the semidihedral group $QD_{32}$. We will show that the group $G \defeq QD_{32} *_{Q_{16}} QD_{32}$ has the property that $\widetilde{K_0 }\mathbb{Z} G$ maps onto a summand $\mathbb{Z}/2$ sitting inside $\widetilde{K_0 }\mathbb{Q} G$. The group $G$ is not special in this regard. The reason for choosing it is that $G$ contains the group $Q_{16}$ as a maximal finite normal subgroup. The group $Q_{16}$ is the smallest group with torsion in negative $K$-theory, hence we do not expect counterexamples of groups which are particularly simpler than $G$. All computations have been done using the computer algebra system $GAP$, \cite{GAP4}.

\subsection{The group $Q_{16}$} \label{generalizedquaternion16}
We let a presentation of $Q_{16}$ be given as
$$ Q_{16} = \left\langle r,s | r^8 = 1,  r^4 = s^2, srs^{-1} = r^7 \right\rangle. $$
It has the following conjugacy classes:
\begin{center}
\begin{tabular}{| c | c c c c |}
\hline
 Class & $\set{1}$ & $\set{ s^2 }$ & $\set{r^2,r^6}$ & $\set{s, r^2 s, r^4 s, r^6 s}$ \\ 
 \hline
 Order & 1 & 2 & 4 & 4  \\  
 Size  & 1 & 1 & 2 & 4  \\
 \hline
\hline
 Class  & $\set{rs, r^3s, r^5s, r^7s}$ & $\set{r,r^7}$ & $\set{r^3,r^5}$ & \\ 
 \hline
 Order & 4 & 8 & 8 & \\  
 Size  & 4 & 2 & 2 & \\
 \hline
\end{tabular}
\end{center}

Since $Q_{16}$ is a $2$-group, by Remark \ref{pgroupshaveonlytorsion} the group $K_{-1} \mathbb{Z} Q_{16}$ must be torsion. Using the ``wedderga'' package in $GAP$, \cite{Wedderga}, we can check the Schur indices appearing in the Wedderburn decomposition of $\mathbb{Q} Q_{16}$. In the following, input commands are prefaced by {\tt Input>}, whereas output produced by $GAP$ is prefaced by {\tt Output>}. \\

\vspace{-2ex}  
\noindent  {\tt  Input> G := QuaternionGroup(16); \\
Input> WedderburnDecompositionWithDivAlgParts( GroupRing( Rationals, G ) );  \vspace{1ex} \newline
Output>
[ [ 1, Rationals ], [ 1, Rationals ], [ 1, Rationals ], [ 1, Rationals ], [ 2, Rationals ],
  [ 1, rec( Center := NF(8,[ 1, 7 ]), DivAlg := true, LocalIndices := [ [ infinity, 2 ] ], SchurIndex := 2 ) ] ]
 }\vspace{1ex}

A few comments on how to read this output are needed. As described in Section \ref{lowerktheory}, the group algebra $\mathbb{Q}G$ splits as
$$ \mathbb{Q} G \cong \prod_{I \in \text{Irr}_\mathbb{Q}(G) } M_{n_I \times n_I} (D_I), $$
with the $D_I$ being finite dimensional division algebras over $\mathbb{Q}$. The function WedderburnDecompositionWithDivAlgParts returns a list containing information about each part $M_{n_I \times n_I} (D_I)$ appearing in the Wedderburn decomposition. First, we have 6 entries corresponding to the 6 irreducible representations of $Q_{16}$. The first number in each of the entries refers to the number $n_I$. Next to it is information about $D_I$. In our case the first $5$ entries happen to have $D_I = \mathbb{Q}$. For the last entry, its division algebra $D$ is non-commutative, which is signalled by {\tt DivAlg := true}. The center $A$ of $D$ is a finite field extension of $\mathbb{Q}$ and described as {\tt NF(8,[ 1, 7 ])}. This notation means that $A$ is a sub-field of the cyclotomic field extension $\mathbb{Q}(\zeta_8)$ being fixed by the subgroup $\set{\underline{1}, \underline{7}}$ of the Galois group $(\mathbb{Z}/8)^\times = \set{\underline{1}, \underline{3}, \underline{5}, \underline{7}}$. 
It is not difficult to see that $A = \mathbb{Q}(\sqrt{2})$, using the code: \\

\vspace{-2ex} \begin{raggedright} 
\noindent 
{\tt 
Input> A := NF( 8, [ 1 , 7 ] ); \\ \vspace{1ex}
Output> NF(8,[ 1, 7 ]) \\ \vspace{1ex}
Input> Dimension(A); \\ \vspace{1ex}
Output> 2 \\ \vspace{1ex}
Input> Sqrt(2) in A; \\ \vspace{1ex}
Output> true \\ \vspace{1ex}
} 
\end{raggedright}
\noindent This means we have the decomposition
$$ \mathbb{Q} Q_{16} \cong \mathbb{Q} \times \mathbb{Q} \times \mathbb{Q} \times \mathbb{Q} \times M_{2 \times 2}(\mathbb{Q}) \times D. $$
The entry {\tt SchurIndex} gives the global Schur index of the representation $I$ and is displayed only when it is bigger than 1. {\tt LocalIndices} gives a list of all primes at which the local Schur index of $I$ is not equal to 1, together with the real Schur index for the value {\tt infinity}.

In our case we can see that $\mathbb{Q} Q_{16}$ has a single irreducible rational representation $\alpha$ with endomorphism algebra $D$, together with a (unique) irreducible $\mathbb Q_2$-representation $\beta$ of $Q_{16}$ such that $\alpha \otimes_\mathbb{Q} \mathbb{Q}_2 = 2 \beta$. We thus have that
$$K_{-1} \mathbb{Z}Q_{16} \cong \mathbb{Z}/2$$
generated by the image of the singular character of $\beta$ under the map
$$\mathrm{SC}(Q_{16}) \rightarrow K_{-1} \mathbb{Z}Q_{16}.$$
The representation $\alpha$ is concretely given by the action of $Q_{16}$ on the quaternion algebra $\mathbb{H}_{\mathbb{Q}(\sqrt{2})} \defeq \mathbb{Q}(\sqrt{2}) \left\langle i , j | i^4 = j^4 = -1, ij = -ji \right\rangle $ over the field $\mathbb{Q}(\sqrt{2})$, realized by 
$$\begin{cases}
r \mapsto (\frac{1}{\sqrt{2}} + \frac{i}{\sqrt{2}}) \\ 
s \mapsto j.
\end{cases}$$
acting via left multiplication on $\mathbb{H}_{\mathbb{Q}(\sqrt{2})}$. In particular, it is $8$-dimensional. Note that $\alpha$ can also be characterized as the unique faithful irreducible $\mathbb{Q}$-representation of $Q_{16}$.

\subsection{The group $QD_{32}$}
A presentation of $QD_{32}$ is given as
$$ QD_{32} = \left\langle a,b | a^{16} = 1, b^2=1 , bab = a^7 \right\rangle. $$

It is easy to see that every element of $QD_{32}$ can be represented in the form $a^n b^i$ for $n = 0, \dots , 15$ and $ i = 0, 1$, from which it follows that $QD_{32}$ has in fact 32 elements. The inclusion $Q_{16} \rightarrow QD_{32}$ can be realized by sending $r \mapsto a^2, s \mapsto a b$ as seen by the calculations
$$ (a^2)^4 = a^8 = a (a^7 b) b = a (b a) b = (ab)^2 $$
as well as
$$ (ab) a^2 = a (a^{2*7}) b = (a^2)^7 (ab). $$
The image of this homomorphism consists of all $a^n b^i$ for which $n+i$ is even, of which there are exactly 16 elements from which it follows that it actually is an inclusion.

The conjugacy classes are given as follows
\begin{center}
\begin{tabular}{| c | c c c c c c |}
\hline
 Class & $\set{1}$ & $\set{ a^8 }$ & $\set{a^{2n} b }$ & $\set{a^4,a^{12}}$& $\set{a^{2n+1}b}$ & $\set{a^2, a^{14}}$  \\ 
 \hline
 Order & 1 & 2 & 2 & 4 & 4 & 8   \\  
 Size  & 1 & 1 & 8 & 2 & 8 & 2   \\
 \hline \hline
 Class & $\set{a^6,a^{10}}$  & $\set{a,a^7}$ & $\set{a^3,a^5}$ & $\set{a^9,a^{15}}$  & $\set{a^{11},a^{13}}$ & \\
 \hline
  Order & 8 & 16 & 16 & 16 & 16 & \\  
 Size   & 2 &  2  &  2 &  2 &  2 & \\
 \hline
\end{tabular}
\end{center}

Similarly to before, $QD_{32}$ is a $2$-group, so $K_{-1} \mathbb{Z} QD_{32}$ is torsion. Doing the same computation of the Schur indices appearing in the Wedderburn decomposition of $QD_{32}$ we get: \\

\vspace{-2ex} \noindent \begin{raggedright} 
{\tt \hspace{-1.5ex} Input> G := SmallGroup(32,19); \vspace{1ex} \\  
Output> <pc group of size 32 with 5 generators> \vspace{1ex} \\  
Input> WedderburnDecompositionWithDivAlgParts( GroupRing( Rationals, G));  \vspace{1ex} \newline 
Output> {[ [ 1, Rationals ], [ 1, Rationals ], [ 1, Rationals ], [ 1, Rationals ], [ 2, Rationals ], [ 2, NF(8,[ 1, 7 ]) ],
  [ 2, NF(16,[ 1, 7 ]) ] ]} \\  
} \vspace{1ex}
\end{raggedright}
\noindent The values $(32,19)$ refer to the ID of $QD_{32}$ in the SmallGroups library of $GAP$. Similarly to before, this means we have the Wedderburn decomposition
$$ \mathbb{Q}QD_{32} \cong \mathbb{Q} \times \mathbb{Q} \times \mathbb{Q} \times \mathbb{Q} \times M_{2 \times 2}(\mathbb{Q}) \times M_{2 \times 2}(A_1) \times M_{2 \times 2}(A_2), $$
with $A_1$ being the sub-field of $\mathbb{Q}(\zeta_8)$ fixed by $\set{\underline{1}, \underline{7}} \subset ( \mathbb{Z}/8 )^\times$ and $A_2$ being the sub-field of $\mathbb{Q}(\zeta_{16})$ fixed by $\set{\underline{1}, \underline{7}} \subset ( \mathbb{Z}/16 )^\times$.
From this we can see that no irreducible rational representations contribute to torsion in $K_{-1} \mathbb{Z} QD_{32}$. Hence
$$K_{-1} \mathbb{Z} QD_{32} = 0.$$
  
\subsection{The group $QD_{32} *_{Q_{16}} QD_{32}$}
The group we want to consider is the group 
$$G \defeq QD_{32} *_{Q_{16}} QD_{32}.$$ 
A concrete presentation is given by
$$
G = \left\langle a, b, a', b' \middle| \begin{matrix}a^{16} = 1, b^2=1 , aba^{-1} = a^7, '^{16} = 1, b'^2=1 , \\
a a'b'a'^{-1} = a'^7, a^2 = a'^2, ab = a' b'\end{matrix} \right\rangle.
$$

The group $G$ is virtually cyclic of type $2$, which means that we can apply the formulas from Section \ref{virtuallycyclictype2}. We have the long exact sequence
$$ 0 \rightarrow \text{ker}^{\widetilde{K_0 }\mathbb Q}  \rightarrow \text{ker}^{\text{SC}} \rightarrow \text{ker}^{K_{-1} \mathbb{Z}} \rightarrow \text{im}( \widetilde{K_0 }\mathbb{Z}G \rightarrow \widetilde{K_0 }\mathbb{Q}G ) \rightarrow 0. $$

The previous calculations show that $K_{-1} \mathbb{Z} Q_{16} = \mathbb{Z}/2$ and  $K_{-1} \mathbb{Z} QD_{32} = 0$, which gives $\text{ker}^{K_{-1} \mathbb{Z}} = \mathbb{Z}/2$. We claim that the map $\text{ker}^{K_{-1} \mathbb{Z}} \rightarrow \text{im}( \widetilde{K_0 }\mathbb{Z}G \rightarrow \widetilde{K_0 }\mathbb{Q}G )$ is injective, which is equivalent to the map $\text{ker}^{\text{SC}} \rightarrow \text{ker}^{K_{-1} \mathbb{Z}}$ being trivial. Since $Q_{16}$ and $QD_{32}$ are $2$-groups, we have isomorphisms $\text{SC}(Q_{16}) \cong \widetilde{K_0 }\mathbb{Q}_2Q_{16}$ and similarly $\text{SC}(QD_{32}) \cong \widetilde{K_0 }\mathbb{Q}_2QD_{32}$ by Remark \ref{scforpgroups}. This means that
\begin{align*}
\text{ker}^{\text{SC}} \cong  \text{ker}( \widetilde{K_0 }\mathbb{Q}_2Q_{16} \xrightarrow{(\iota_1,-\iota_2)} \widetilde{K_0 }\mathbb{Q}_2QD_{32} \oplus \widetilde{K_0 }\mathbb{Q}_2QD_{32} ) \\
 = \text{ker}( \widetilde{K_0 }\mathbb{Q}_2Q_{16} \rightarrow \widetilde{K_0 }\mathbb{Q}_2QD_{32} ).
\end{align*}
By Corollary \ref{splittingofplocalization} the map $\widetilde{K_0 }\mathbb{Q} Q_{16} \rightarrow \widetilde{ K_0 }\mathbb{Q}_2 Q_{16}$ splits as
$$\bigoplus_{ I \in \text{Irr}_\mathbb{Q}(G)} \mathbb{Z} \rightarrow \bigoplus_{ I \in \text{Irr}_\mathbb{Q}(G)} \mathbb{Z} \set{ K \in \text{Irr}_{\mathbb{Q}_2}(Q_{16}) | K ~\text{appears as a summand in}~ I\otimes \mathbb{Q}_2 } $$
As discussed earlier, the group $Q_{16}$ has a single irreducible $8$-dimensional $\mathbb{Q}$-representation $\alpha$ together with a (unique) irreducible $4$-dimensional $\mathbb Q_2$-representation $\beta$ of $Q_{16}$ such that $\alpha \otimes_\mathbb{Q} \mathbb{Q}_2 = 2 \beta$. The negative $K$-theory group $K_{-1} \mathbb{Z} Q_{16} = \mathbb{Z}/2$ is generated by the image of $\beta$. Neither $\alpha$ nor $\beta$ can lie in the kernels of $\iota_1$ and $\iota_2$, respectively, since their inductions to $QD_{32}$ are neither the trivial nor regular representations (by looking at their dimensions), which shows the claim.

In summary, we have just shown that for the group $G$,
$$\text{im} ( \widetilde{K_0 }\mathbb{Z} G  \rightarrow \widetilde{K_0 }\mathbb{Q} G ) \cong \mathbb{Z}/2.$$

\subsection{Other examples}

The group $QD_{32}$ is not special beyond the property that the map $K_{-1} \mathbb{Z} Q_{16} \rightarrow K_{-1} \mathbb{Z} QD_{32}$ is not injective. In fact, the group $Q_{16}$ sits inside $5$ different groups of order $32$. This can be checked with the $GAP$ code: \\

\vspace{-2ex} \noindent \begin{raggedright} 
{\tt \hspace{-1.5ex}
Input> for G in AllSmallGroups(32) do \\
> if ForAny( NormalSubgroups(G) , H -> IdSmallGroup(H) = [ 16, 9 ] ) then Print(IdSmallGroup(G)); fi; \\
> od;\\ 
Output> {[ 32, 19 ][ 32, 20 ][ 32, 41 ][ 32, 42 ][ 32, 44 ]} \\
} \end{raggedright}
Here the value $( 16, 9 )$ refers to the ID of $Q_{16}$ in the SmallGroups library in $GAP$.

We will analyse them case by case:
\begin{enumerate}
\item ID = [ 32, 19 ], also known as $QD_{32}$: As discussed $K_{-1} \mathbb{Z} QD_{32} = 0$.
\item ID = [ 32, 20 ], also known as $Q_{32}$: Here we can show that $K_{-1} \mathbb{Z} Q_{16} \rightarrow K_{-1} \mathbb{Z} Q_{32}$ induces an isomorphism.
\item ID = [ 32, 41 ], also known as $Q_{16} \times C_2$: Here we can show that $K_{-1} \mathbb{Z} (Q_{16} \times C_2) = (\mathbb{Z}/2)^2$ and the map  $K_{-1} \mathbb{Z} Q_{16} \rightarrow K_{-1} (Q_{16} \times C_2)$ corresponds to the diagonal $\mathbb{Z}/2 \rightarrow (\mathbb{Z}/2)^2$. In particular it is injective.
\item ID = [ 32, 42 ], also known as $C_4 \circ D_8$: This group has  $K_{-1} \mathbb{Z} C_4 \circ D_8 = 0$.
\item ID = [ 32, 44 ], also known as $C_8.C_2^2$: Here we can show that the map $K_{-1} \mathbb{Z} Q_{16} \rightarrow K_{-1} \mathbb{Z} (C_8.C_2^2)$ induces an isomorphism.
\end{enumerate}

In summary, the only virtually cyclic groups of type 2 that contain $Q_{16}$ as kernel for which $ \text{im} ( \widetilde{K_0 }\mathbb{Z} G  \rightarrow \widetilde{K_0 }\mathbb{Q} G ) $ is non-trivial are the groups
$$ QD_{32} \ast_{Q_{16}} QD_{32},     QD_{32}\ast_{Q_{16}} (C_4 \circ D_8), ~\text{and}~     (C_4 \circ D_8) \ast_{Q_{16}} (C_4 \circ D_8). $$
In each of those cases we have
$$\text{im} ( \widetilde{K_0 }\mathbb{Z} G  \rightarrow \widetilde{K_0 }\mathbb{Q} G ) = \mathbb{Z}/2.$$

\section{Comparison to related functors}

%
%
%

We can ask if there is a more general statement to the one considered in this paper on $K_0$ for higher $K$-groups. Here the most natural way to generalize to $\pi_1$ would be to understand the map
$$ \text{Wh}_1(\mathbb{Z};G) \rightarrow \text{Wh}_1(\mathbb{Q};G). $$

For finite $G$, we have that 
$$\text{Wh}_1(\mathbb{Z};G) \cong \mathbb{Z}^{r_\mathbb{R}-r_\mathbb{Q}} \oplus SK_1(\mathbb{Z}G),$$
where $r_\mathbb{R}$ and $r_\mathbb{Q}$ are the number of real and rational representations, respectively, and $SK_1(\mathbb{Z}G)$ is a finite group, given as the kernel of the map $K_1(\mathbb{Z}G) \rightarrow K_1(\mathbb{Q}G)$. (See \cite{oliver_1988}, page 6)

Using the groups of type $G = H \times (C_\infty)^2$, the Bass-Heller-Swan decomposition already tells us that any defects of the maps in $K_{-1}$, $\widetilde{K_0}$ and $\text{Wh}_1$ for virtually cyclic groups $H$ will enter the picture. Thus we can easily find counterexamples to the possibility that for example the map $\text{Wh}_1(\mathbb{Z};G) \rightarrow \text{Wh}_1(\mathbb{Q};G)$ is an isomorphism rationally (injectivity fails for the corresponding statement in $K_{-1}$ for $H$ being any group $H$ with an element of non-prime power order and surjectivity fails for $\widetilde{K_0}$ for any non-trivial group $H$) or that it kills all torsion (not true in $\widetilde{K_0}$ by the results of this paper).

\appendix

\section{Stable $\infty$-categories and $\lowercase{t}$-structures} \label{tstructures}

In this appendix we develop some of the tools for dealing with $t$-structures on stable $\infty$-categories. The standard reference will be Section 1.2.1 in Lurie, \cite{lurieha}. We note that if $\mathcal{C}$ is a stable $\infty$-category, its homotopy category is naturally a triangulated category. The notion of a $t$-structure has been first developed for triangulated categories in \cite{bbd} and we will give a definition here. It is worth pointing out that \cite{bbd} uses cohomological indexing, whereas we follow the homological indexing used in \cite{lurieha}.

\begin{definition} Let $\mathcal D$ be a triangulated category. A $t$-structure on $\mathcal D$ is a pair of full subcategories $\mathcal D_{\geq 0}$ and $\mathcal D_{\leq 0}$, both closed under isomorphisms, such that the following three conditions hold. Here $\mathcal D_{\geq n} \defeq \Sigma^n \mathcal D_{\geq 0}$ and $\mathcal D_{\leq n} \defeq \Sigma^n \mathcal D_{\leq 0}$ are defined as the essential images under the functors $\Sigma^n$ for all $n \in \mathbb{Z}$.
\begin{itemize}
\item If $X \in \mathcal D_{\geq 0}$, $Y \in \mathcal D_{\leq -1}$ then $\text{Hom}_{\mathcal D}(X,Y) = 0$.
\item $\mathcal D_{\geq 1} \subset \mathcal D_{\geq 0}$, $\mathcal D_{\leq -1} \subset \mathcal D_{\leq 0}$
\item For all objects $X$ in $\mathcal D$ we have a distinguished triangle
$$ X' \rightarrow X \rightarrow X'' \rightarrow \Sigma X' $$
with $X' \in \mathcal D_{\geq 0}$ and $X'' \in \mathcal D_{\leq -1}$.
\end{itemize}
\end{definition}

A $t$-structure on a stable $\infty$-category is defined to be a $t$-structure on its homotopy category $h \mathcal{C}$. We define two types of subcategories of $\mathcal{C}$, namely $\mathcal{C}_{\geq n}$ and $\mathcal{C}_{\leq n}$, as the full subcategories of $\mathcal{C}$ corresponding to the subcategories $h\mathcal{C}_{\geq n}$ and $h\mathcal{C}_{\leq n}$ of $h \mathcal{C}$, respectively. The inclusions of the subcategories $\mathcal{C}_{\geq n}$ in $\mathcal{C}$ admit right adjoints denoted by $\tau_{\geq n}$, which act as the identity when restricted to $\mathcal{C}_{\geq n}$. Consequently, $\mathcal{C}_{\geq n}$ is closed under colimits in $\mathcal{C}$. Dually, the subcategories $\mathcal{C}_{\leq n}$ in $\mathcal{C}$ admit left adjoints $\tau_{\geq n}$, which act as the identity on $\mathcal{C}_{\leq n}$, and $\mathcal{C}_{\leq n}$ is closed under limits in $\mathcal{C}$ (\cite{lurieha} Proposition 1.2.1.5.). The compositions $\tau_{\geq n} \circ \tau_{\leq m}$ and $\tau_{\leq m} \circ \tau_{\geq n}$ are naturally equivalent (\cite{lurieha} Proposition 1.2.1.10.) and will be denoted as $A \mapsto A[m,n]$, or $A \mapsto A[n]$ in the case $n=m$.\footnote{Our choice of notation clashes here with the one used in \cite{bbd}, where $X[n]$ denotes the $n$-th suspension of $X$.} The intersection $\mathcal{C}^\heartsuit \defeq \mathcal{C}_{\geq 0} \cap \mathcal{C}_{\leq 0}$ is called the heart of $\mathcal{C}$ and is (equivalent to the nerve of) an abelian $1$-category. There are functors $\pi_n \defeq ( A \mapsto (\Sigma^{-n} A) [0])$, from $\mathcal{C}$ to $\mathcal{C}^\heartsuit$ which will be called homotopy group functors. (\cite{lurieha}  
Definition 1.2.1.11. and Remark 1.2.1.12.)

\subsection{Homological algebra in the setting of $t$-structures}

The following section is concerned with the relationship between computations involving fiber sequences in $\mathcal{C}$ and homological algebra in the abelian category $\mathcal{C}^\heartsuit$. The following theorem is the central part of this section: Fiber sequences in $\mathcal C$ give rise to long exact sequences in $\mathcal{C}^\heartsuit$. 

\begin{theorem}[See \cite{bbd}, Theorem 1.3.6] \label{longexactsequence}
Let $\mathcal C$ be a stable $\infty$-category with a $t$-structure. Let
$$X \xrightarrow{f} Y \xrightarrow{g} Z$$
be a fiber sequence. Then there is an induced long exact sequence
$$ \cdots \rightarrow \pi_{n+1} Z \rightarrow \pi_n X \xrightarrow{\pi_n f} \pi_n Y \xrightarrow{\pi_n g} \pi_n Z  \rightarrow \pi_{n-1} X \rightarrow \cdots $$
where the maps $\pi_n Z  \rightarrow \pi_{n-1} X$ come from $\pi_n$ applied to the boundary map $Z \rightarrow \Sigma X$ which realizes the cofiber of $X \rightarrow Y$.
\end{theorem}

Next, we are concerned with degree $1$ maps between objects in the heart of $\mathcal C$.

\begin{lemma} \label{mapsofdegreeoneandexactsequence}
Let $\mathcal C$ be a stable $\infty$-category with a $t$-structure and let $A, C$ be two objects in the heart $\mathcal{C}^\heartsuit$. There is a natural isomorphism
$$\phi  \colon  [C , \Sigma A] \cong \text{Ext}^1_{\mathcal{C}^\heartsuit}(C, A) $$
where
$$\phi(\beta  \colon  C \rightarrow \Sigma A) = ( 0 \rightarrow A \rightarrow \text{fib}(\beta) \rightarrow C \rightarrow 0 ).$$
\end{lemma}

\begin{proof}
To show that $\phi$ is well-defined, we still have to show that $\text{fib}(\beta)$ lies in the heart of $\mathcal{C}$. To do so, note that by Theorem \ref{longexactsequence} we have the long exact sequence in homotopy groups
$$ \cdots \rightarrow \pi_1(C) \rightarrow \pi_1( \Sigma A ) \rightarrow \pi_0 ( \text{fib}(\beta) ) \rightarrow \pi_0(C) \rightarrow \pi_0( \Sigma A ) \rightarrow \cdots. $$
Since $A$ and $C$ are in the heart, we have $\pi_1(C) = 0$, $\pi_0( \Sigma A ) = 0$, which shows that $\text{fib}(\beta)$ lies in the heart. Furthermore, $\pi_0(C) = C$ and $\pi_1( \Sigma A ) = A$, which means we do, in fact, get the claimed exact sequence.

The inverse map $\psi  \colon  \text{Ext}^1_{\mathcal{C}^\heartsuit}(C, A) \rightarrow [C , \Sigma A]$ is constructed as follows.
A given exact sequence
$$ 0 \rightarrow A \rightarrow B \rightarrow C \rightarrow 0$$
in the heart produces a fiber sequence
$$ A \rightarrow B \rightarrow C $$
in $\mathcal C$ which can be mapped to the boundary map $ \delta  \colon  C \rightarrow \Sigma A$. It is clear that that the two processes are mutually inverse.
\end{proof}

The following lemma clarifies an argument used multiple times during the main part of this paper.

\begin{lemma} \label{usefullemma}
Let $\mathcal C$ be a stable $\infty$-category with a $t$-structure and let $A$ be an object of $\mathcal C$ concentrated in degrees $-1$ and $0$ and $B$ an object in the heart $\mathcal{C}^\heartsuit$. Let $f$ be map $A \rightarrow B$ such that the composition $A[0] \rightarrow A \xrightarrow{f} B$ is zero. Then $f$ factorizes through an up to homotopy unique map $\tilde{f}  \colon  A[-1] \rightarrow B$, i.e. we have a commutative triangle
$$\xymatrix{
A \ar[r]^f \ar[d] & B \\
A[-1] \ar[ru]^{\tilde{f}} &
}$$
\end{lemma}

\begin{proof}
We have the fiber sequence 
$$ A[0] \rightarrow {A} \rightarrow {A}[-1], $$
which implies the long exact sequence
$$ [ \Sigma {A}[0], {B} ] \rightarrow  [ {A}[-1], {B} ] \rightarrow [ {A}, {B} ] \rightarrow  [ {A}[0], {B} ]. $$
The abelian group $[ \Sigma {A}[0], {B} ]$ is zero since $\Sigma {A}[0]$ is $1$-connected and ${B}$ was assumed to be concentrated in degree $0$, so if $f  \colon  A \rightarrow B$ is a map that becomes the zero map when precomposed with $ {A}[0] \rightarrow {A}$, it factors through a map $\tilde{f} \colon  {A}[-1] \rightarrow {B} $ which is unique up to homotopy.
\end{proof}

The last part of this section is concerned with a technical lemma about the relationship between the induced map on fibers coming from a square involving degree $1$ maps and the well known connecting map from the snake lemma in homological algebra.

\begin{lemma} \label{tsnakelemma}
Let $\mathcal C$ be a stable $\infty$-category with a $t$-structure and let $A_1, A_2, C_1$ and $C_2$ be objects in the heart $\mathcal{C}^\heartsuit$ and suppose we have a commutative square
$$\xymatrix{
C_1 \ar[r]^{\beta_1} \ar[d]^{f_C} & \Sigma A_1 \ar[d]^{\Sigma f_A} \\
C_2 \ar[r]^{\beta_2}        & \Sigma A_2 
}$$
Write $\text{fib}_A \defeq \text{fib}( f_A )$, $\text{cok}_A \defeq \pi_{-1} \text{fib}_A$, $\text{fib}_C \defeq \text{fib}( f_C )$, and $\text{ker}_C \defeq \pi_0 \text{fib}_C$. Then the induced map $\pi_0 \text{fib}_C \rightarrow \pi_0 \Sigma \text{fib}_A$ agrees with the map $\delta  \colon  \text{ker}_C \rightarrow \text{cok}_A$ induced by the snake lemma for the corresponding map of exact sequences
$$\xymatrix{
0 \ar[r] & A_1 \ar[r] \ar[d] & B_1 \ar[r] \ar[d] & C_1 \ar[r] \ar[d] & 0 \\
0 \ar[r] & A_2 \ar[r]        & B_2 \ar[r]        & C_2 \ar[r]        & 0
}$$
in the heart $\mathcal{C}^\heartsuit$.
\end{lemma}

Before we begin the proof of Lemma \ref{tsnakelemma}, we want to establish some facts about the map induced by the well known snake lemma (see e.g. \cite{weibel1995introduction}, Lemma 1.3.2).

\begin{lemma} \label{originalsnake}
Suppose $\mathcal{A}$ is an abelian category and
$$\xymatrix{
0 \ar[r] & A_1 \ar[r]^{g_1} \ar[d]^{f_A} & B_1 \ar[r]^{\pi_1} \ar[d]^{f_B} & C_1 \ar[r] \ar[d]^{f_C} & 0 \\
0 \ar[r] & A_2 \ar[r]^{g_2}        & B_2 \ar[r]^{\pi_2}        & C_2 \ar[r]        & 0
}$$
is a diagram in $\mathcal{A}$ with exact rows. Write $\text{ker}_C \defeq \text{ker}(f_C)$ and $\text{cok}_A \defeq \text{cok}(f_A)$. Then:
\begin{enumerate}
\item There is an isomorphism
$$ \theta  \colon  \frac{\text{ker}( B_1 \oplus A_2 \xrightarrow{(g_2, -f_B)} B_2)}{\text{im}( A_1 \xrightarrow{(g_1 , f_A)} B_1 \oplus A_2 ) } \cong \text{ker}_C$$
induced by the composition $B_1 \oplus A_2 \rightarrow B_1 \xrightarrow{\pi_1} C_1$, where the first map is the projection onto the first summand.
\item There is a natural map 
$$ \phi  \colon  \frac{\text{ker}( B_1 \oplus A_2 \xrightarrow{(g_2, -f_B)} B_2)}{\text{im}( A_1 \xrightarrow{(g_1 , f_A)} B_1 \oplus A_2 ) } \rightarrow \text{cok}_A$$
induced by the composition $B_1 \oplus A_2 \rightarrow A_2 \rightarrow \text{cok}_A$, where the first map is projection onto the second summand.
\item The composition $ \phi \theta^{-1}  \colon  \text{ker}_C \rightarrow \text{cok}_A$ is the natural connecting map from the snake lemma.
\item If $\pi_1$ and $\pi_2$ have sections $s_1  \colon  C_1 \rightarrow B_1$, $s_2  \colon  C_2 \rightarrow B_2$ such that
$$\xymatrix{
B_1   \ar[d]^{f_B} & C_1 \ar[l]^{s_1} \ar[d]^{f_C}\\
B_2                & C_2 \ar[l]^{s_2}
}$$
commutes then $ \phi \theta^{-1} = 0$.
\end{enumerate}
\end{lemma}

\begin{proof}
To show that $\theta$ is well-defined, we need to show:
\begin{itemize}
\item
The composition
$$ A_1 \xrightarrow{(g_1 , f_A)} \text{ker}( B_1 \oplus A_2 \xrightarrow{(g_2, -f_B)} B_2) \rightarrow B_1 \xrightarrow{\pi_1} C_1 $$
is the trivial map. This is simply because $\pi_1 g_1 = 0$. From this follows that the map 
$$\text{ker}( B_1 \oplus A_2 \xrightarrow{(g_2, -f_B)} B_2) \rightarrow B_1 \xrightarrow{\pi_1} C_1$$
factors through $\text{ker}( B_1 \oplus A_2 \xrightarrow{(g_2, -f_B)} B_2) / \text{im}( A_1 \xrightarrow{(g_1 , f_A)} B_1 \oplus A_2 ) $.
\item
The composition
$$ \text{ker}( B_1 \oplus A_2 \xrightarrow{(g_2, -f_B)} B_2) \rightarrow B_1 \xrightarrow{\pi_1} C_1 \xrightarrow{f_C} C_2 $$
is the trivial map. We have the following equalities of maps
\begin{align*}
 \text{ker}( B_1 \oplus A_2 \xrightarrow{(g_2, -f_B)} B_2) \rightarrow B_1 \xrightarrow{\pi_1} C_1 \xrightarrow{f_C} C_2  \\[-11pt] \cline{1-1}
 = \text{ker}( B_1 \oplus A_2 \xrightarrow{(g_2, -f_B)} B_2) \rightarrow B_1 \xrightarrow{f_B} B_2 \xrightarrow{\pi_2} C_2  \\[-11pt] \cline{1-1}
 = \text{ker}( B_1 \oplus A_2 \xrightarrow{(g_2, -f_B)} B_2) \rightarrow A_2 \xrightarrow{g_2} B_2 \xrightarrow{\pi_2} C_2  \\[-11pt]
\end{align*}	
The claim now follows since $\pi_2 g_2 = 0$. From this follows that the map 
$$\text{ker}( B_1 \oplus A_2 \xrightarrow{(g_2, -f_B)} B_2) \rightarrow B_1 \xrightarrow{\pi_1} C_1$$
maps into $\text{ker}_C$.
\end{itemize}
To see that $\theta$ is an isomorphism, we show two things:
\begin{itemize}
\item $\text{ker}(\theta) = 0$. This is because
$$ \text{ker}( B_1 \oplus A_2 \rightarrow B_1 \rightarrow C_1 ) = A_1 \oplus A_2 $$
which implies
\begin{align*}
\text{ker}( B_1 \oplus A_2 \rightarrow B_1 \rightarrow C_1 ) \cap \text{ker}( B_1 \oplus A_2 \xrightarrow{(g_2, -f_B)} B_2) \\
 = \text{im}( A_1 \xrightarrow{ ( g_1 , f_A ) } B_1 \oplus A_2 )
\end{align*}
Hence, $\text{ker}(\theta) = 0$.
\item $\text{cok}(\theta) = 0$. A simple diagram chase using exactness at $B_2$ shows that
$$ \text{im}( \text{ker}( B_1 \oplus A_2 \rightarrow B_2) \rightarrow B_1) = \text{ker}( \pi_2 f_B) $$
Hence, since $B_1 \rightarrow C_1$ is an epimorphism, 
$$ \text{im}( \text{ker}( B_1 \oplus A_2 \rightarrow B_2) \rightarrow B_1 \rightarrow C_1) = \text{ker}_C $$
From this follows that $\theta$ is an epimorphism.
\end{itemize}
We now need to show that $\phi$ is well-defined. To do so, we need to show that the composition
$$A_1 \xrightarrow{ (g_1 , f_A) } B_1 \oplus A_2 \rightarrow A_2 \rightarrow \text{cok}_A$$
is trivial. This is clear, however, as the composition $A_1 \xrightarrow{(g_1 , f_A)} B_1 \oplus A_2 \rightarrow A_2$ is just equal to $f_A$.

The next claim is that the map $\phi \theta^{-1}$ agrees with the map induced by the snake lemma. For simplicity, assume that $\mathcal{A}$ is the category of abelian groups.\footnote{An element-free proof can be done, of course. Our proof is sufficient by the Freyd-Mitchell embedding theorem.} The traditional way of defining the boundary map $\delta$ goes as follows. Assume $c \in \text{ker}_C$. Using surjectivity of $B_1 \rightarrow C_1$, find a preimage $b_1 \in B$ of $c$. Since $\pi_2 f_B = f_C \pi_1$, the element $f(b_1)$ lies in the kernel of $\pi_2$; hence, there is a unique $a_2 \in A_2$ such that $g_2(a_2) = f(b_1)$. The image of $\delta$ of the element $c$ is defined as the class of $a_2$ in the cokernel $\text{cok}_A$. The reason this agrees with $\phi \theta^{-1}$ is as follows. The class $[b_1 , a_2]$ is just a preimage of $c$ under the map $\theta$ and the assignment $[b_1 , a_2] \rightarrow [a_2]$ is exactly what defines the map $\phi$.

Lastly, assume $\pi_1$ and $\pi_2$ have commuting sections $s_1$ and $s_2$ respectively. Then the map $C_1 \xrightarrow{(s_1 , 0) } B_1 \oplus A_2$ descends to the inverse of $\theta$. It is then clear that $ \phi \theta^{-1} = 0$ since $\phi$ is induced by projection on the $A_2$ coordinate.
\end{proof}

%

We still need to introduce some new terminology. Suppose  $D$ is a commutative square
$$\xymatrix{
X_1 \ar[r]^{g_1} \ar[d]^{f_X} & Y_1 \ar[d]^{f_Y} \\
X_2 \ar[r]^{g_2}  & Y_2
}$$
in a stable $\infty$-category $\mathcal{C}$. Define the \emph{total cofiber} of $D$ as
$$\text{cof}(D) \defeq \text{cof}( \text{cof}( g_1 ) \rightarrow \text{cof}( g_2 ) ) \simeq  \text{cof}( \text{cof}( f_X ) \rightarrow \text{cof}( f_Y ) )$$
Define $\Box \defeq \Delta^1 \times \Delta^1$. It is clear that taking total cofibers is functorial in the sense that it defines an exact functor
$$ \text{cof}  \colon  \text{Fun}(\Box, \mathcal{C}) \rightarrow \mathcal{C}, $$
which is the left adjoint to the functor that sends an object $X \in \mathcal{C}$ to the square
$$\xymatrix{
0 \ar[r] \ar[d] & 0 \ar[d] \\
0 \ar[r]  & X.
}$$
The reason we are interested in this construction is that if we take the objects $X_i$ and $Y_i$ to be in the heart of a $t$-structure on $\mathcal{C}$, this allows us to model chain complexes of length $\leq 3$ in $\mathcal{C}$.\footnote{This construction generalizes to arbitrary length by defining total cofibers of $n$-cubes in a similar manner.} The following lemma will make this precise.

\begin{lemma}
Suppose $D$ is a commutative square
$$\xymatrix{
A_1 \ar[r]^{g_1} \ar[d]^{f_A} & B_1 \ar[d]^{f_B} \\
A_2 \ar[r]^{g_2}  & B_2
}$$
with values in $\mathcal{C}^\heartsuit$. Then:
\begin{enumerate}
\item The total cofiber \text{cof}(D) is concentrated in degrees $0$,$1$ and $2$. Moreover, we have
\begin{align*}
\pi_2 \text{cof}(D) &= \text{ker}( A_1 \xrightarrow{(g_1 , f_A)} B_1 \oplus A_2 ) \\
\pi_1 \text{cof}(D) &= \frac{\text{ker}( B_1 \oplus A_2 \xrightarrow{(g_2, -f_B)} B_2)}{\text{im}( A_1 \xrightarrow{(g_1 , f_A)} B_1 \oplus A_2 ) } \\
\pi_0 \text{cof}(D) &= \text{cok}( B_1 \oplus A_2 \xrightarrow{(g_2, -f_B)} B_2 ) 
\end{align*}
\item If the square $D$ has the form
$$\xymatrix{
0 \ar[r] \ar[d] & B_1 \ar[d]^{f_B} \\
0 \ar[r]  & B_2,
}$$
then $\text{cof}(D) = \text{cof}(f_B)$.
\item If the square $D$ has the form
$$\xymatrix{
A_1 \ar[r] \ar[d]^{f_A} & 0 \ar[d] \\
A_2 \ar[r]  & 0,
}$$
then $\text{cof}(D) = \Sigma \text{cof}(f_A)$.
\item If the square $D$ has the form
$$\xymatrix{
0  \ar[r] \ar[d] & B_1 \ar[d] \\
A_2 \ar[r]  & 0,
}$$
then $\text{cof}(D) = \Sigma ( A_2 \oplus B_1 )$.
\end{enumerate}
\end{lemma}

\begin{proof}
Point (2) and (3) are trivial. For point (4), note that the space of morphisms $\text{Map}_\mathcal{C}(0,0)$ is contractible, hence the square 
$$\xymatrix{
0  \ar[r] \ar[d] & B_1 \ar[d] \\
A_2 \ar[r]  & 0
}$$
is trivially commutative and taking vertical cofibers realizes to the zero map $ A_2 \rightarrow \Sigma B_1$ which implies that $\text{cof}(D) = \text{cof}( A_2 \xrightarrow{0} \Sigma B_1 ) = \Sigma (A_2 \oplus B_1)$.

Now assume $D$ is of the shape
$$\xymatrix{
0  \ar[r] \ar[d] & B_1 \ar[d] \\
A_2 \ar[r]  & B_2.
}$$
Then we have the following fiber sequence of square diagrams,
$$\begin{tikzcd}[row sep=1.5em, column sep = 1.5em]
   0 \arrow[rr] \arrow[dr] \arrow[dd] && 0 \ar[rr] \ar[dd] \ar[dr] &&
   0 \arrow[dd] \arrow[dr] \\
   & B_1 \arrow[rr] \arrow[dd, "f_B" near start ] && B_1 \arrow[rr] \arrow[dd] &&
   0 \arrow[dd] \\
   A_2 \arrow[rr] \arrow[dr, "g_2"] && A_2        
   \arrow[rr] \arrow[dr] && 0 \arrow[dr] \\
   & B_2 \arrow[rr] && 0 \arrow[rr] && \Sigma B_2.
  \end{tikzcd}$$
Taking vertical cofibers of the right hand cube results in the square
$$\xymatrix{
A_2 \ar[r] \ar[d]^0 & 0 \ar[d] \\   
\Sigma B_1 \ar[r]^{-f_B} & \Sigma B_2
}$$
where the resulting square
$$\xymatrix{
A_2 \ar[r] \ar[d] & 0 \ar[d] \\   
0 \ar[r] & \Sigma B_2
}$$
classifies the map $g_2  \colon  \Sigma A_2 \rightarrow \Sigma B_2$. This means taking further cofibers results in the map
$$\Sigma (A_2 \oplus B_1) \xrightarrow{(g_2, -f_B)} \Sigma B_2.$$
This means we have a fiber sequence
$$ \text{cof}(D) \rightarrow \Sigma (A_2 \oplus B_1) \xrightarrow{(g_2, -f_B)} \Sigma B_2$$
from which we can read off that $\text{cof}(D)$ is concentrated in degree $0$ and $1$ with the homotopy groups
\begin{align*}
\pi_1 \text{cof}(D) & = \text{ker}( A_2 \oplus B_1 \xrightarrow{(g_2, -f_B)} B_2 ) \\
\pi_0 \text{cof}(D) & = \text{cok}( A_2 \oplus B_1 \xrightarrow{(g_2, -f_B)} B_2 ).
\end{align*}

Now assume $D$ is a general commutative square of the form
$$\xymatrix{
A_1 \ar[r]^{g_1} \ar[d]^{f_A} & B_1 \ar[d]^{f_B} \\
A_2 \ar[r]^{g_2}  & B_2.
}$$
We have the following fiber sequence of square diagrams,
$$\begin{tikzcd}[row sep=1.5em, column sep = 1.5em]
   \Omega A_1 \arrow[rr] \arrow[dr] \arrow[dd] && 0 \ar[rr] \ar[dd] \ar[dr] &&
   A_1 \arrow[dd,"f_A" near end] \arrow[dr,"g_1"] \\
   & 0 \arrow[rr] \arrow[dd] && B_1 \arrow[rr] \arrow[dd] &&
   B_1 \arrow[dd] \\
   0 \arrow[rr] \arrow[dr] && A_2        
   \arrow[rr] \arrow[dr] && A_2 \arrow[dr] \\
   & 0 \arrow[rr] && B_2 \arrow[rr] && B_2,
  \end{tikzcd}$$
which produces the following two exact sequences
$$ 0 \rightarrow \pi_2 \text{cof}(D) \rightarrow A_1 \rightarrow \text{ker}( A_2 \oplus B_1 \xrightarrow{(g_2, -f_B)} B_2  ) \rightarrow \pi_1 \text{cof}(D) \rightarrow 0 $$
as well as
$$ 0 \rightarrow \text{cok}( A_2 \oplus B_1 \xrightarrow{(g_2, -f_B)} B_2 ) \rightarrow \pi_0 \text{cof}(D) \rightarrow 0, $$
which proves point (1).
\end{proof}

We are now ready to prove Lemma \ref{tsnakelemma}. Assume now that we have a diagram with short exact rows,
$$\xymatrix{
0 \ar[r] & A_1 \ar[r]^{g_1} \ar[d]^{f_A} & B_1 \ar[r] \ar[d]^{f_B} & C_1 \ar[r] \ar[d]^{f_C} & 0 \\
0 \ar[r] & A_2 \ar[r]^{g_2}        & B_2 \ar[r]        & C_2 \ar[r]        & 0,
}$$
and write $\text{cof}_A \defeq \text{cof}( f_A )$, $\text{cof}_B \defeq \text{cof}( f_B )$,$\text{cof}_C \defeq \text{cof}( f_C )$, $\text{ker}_C \defeq \pi_0 \text{fib}_C$ and $\text{cok}_A \defeq \pi_{-1} \text{fib}_A$.

\begin{proof}[Proof of Lemma \ref{tsnakelemma}.] We now come back to the claim that the map
$$ \text{fib}_C \rightarrow \Sigma \text{fib}_A $$
induces the map described by the snake lemma in $\pi_0$. Note that $\text{cof}_A = \Sigma \text{fib}_A$, and similarly for $C$, so to proof that $ \pi_0 \text{fib}_C \rightarrow \pi_0 \Sigma \text{fib}_A $ is the map induced by the snake lemma, it suffices to show the same thing for $\pi_1$ on the cofibers.

Take the square $D$
$$\xymatrix{
A_1 \ar[r]^{g_1} \ar[d]^{f_A} & B_1 \ar[d]^{f_B} \\
A_2 \ar[r]^{g_2}  & B_2.
}$$
There is a commuting cube
$$\begin{tikzcd}[row sep=1.5em, column sep = 1.5em]
   A_1 \ar[rr] \ar[dd] \ar[dr] &&
   0 \arrow[dd] \arrow[dr] \\
   & B_1 \arrow[rr] \arrow[dd] &&
   C_1 \arrow[dd] \\
   A_2  \arrow[rr] \arrow[dr] && 0 \arrow[dr] \\
   & B_2 \arrow[rr] && C_2.
  \end{tikzcd}$$
Taking total cofibers yields a map $\text{cof}(D) \rightarrow \text{cof}_C$. It is clear that the map induced on $\pi_1$ of this is the map $\theta$ described in lemma \ref{originalsnake}. Moreover, the map in $\pi_0$ is an isomorphism as well and $\pi_2(D) = 0$ since $A_1 \rightarrow B_1$ is injective, hence $\text{cof}(D) \rightarrow \text{cof}_C$ is an equivalence.

Now take the fiber sequence of commutative squares
$$\begin{tikzcd}[row sep=1.5em, column sep = 1.5em]
   0 \arrow[rr] \arrow[dr] \arrow[dd] && A_1 \ar[rr] \ar[dd] \ar[dr] &&
   A_1 \arrow[dd] \arrow[dr] \\
   & B_1 \arrow[rr] \arrow[dd] && B_1 \arrow[rr] \arrow[dd] &&
   0 \arrow[dd] \\
   0 \arrow[rr] \arrow[dr] && A_2        
   \arrow[rr] \arrow[dr] && A_2 \arrow[dr] \\
   & B_2 \arrow[rr] && B_2 \arrow[rr] && 0.
  \end{tikzcd}$$
Taking total cofibers gives the fiber sequence
$$\text{cof}_B \rightarrow \text{cof}(D) \rightarrow \Sigma \text{cof}_A$$
Here it is clear that the map $\text{cof}(D) \rightarrow \Sigma \text{cof}_A$ on $\pi_1$ becomes the map $\phi$ in lemma \ref{originalsnake}. Taking all things together we see that the map
$$ (\text{cof}_C \rightarrow \Sigma \text{cof}_A) \simeq (\text{cof}_C \rightarrow \text{cof}(D) \rightarrow \Sigma \text{cof}_A) $$
gives the map $\phi \theta^{-1}$ on $\pi_1$, which by \ref{originalsnake} is the map induced by the snake lemma.
\end{proof}

\subsection{Functor categories and $t$-structures}
Given a stable $\infty$-category $\mathcal C$ and a small $\infty$-category $\mathcal{D}$, we know that $\text{Fun}(\mathcal{D},\mathcal C)$ is again a stable $\infty$-category. If we have a $t$-structure on $\mathcal{C}$, we can put a natural $t$-structure on $\text{Fun}(\mathcal{D},\mathcal{C})$:

\begin{definition} \label{objectwisetstructure}
Suppose $\mathcal C$ is a stable $\infty$-category with $t$-structure and $\mathcal{D}$ a small $\infty$-category. The object-wise $t$-structure on $\text{Fun}(\mathcal{D},\mathcal{C})$ is defined via
\begin{eqnarray*}
\text{Fun}(\mathcal{D},\mathcal{C})_{\leq 0}  \defeq & \text{Fun}(\mathcal{D},\mathcal{C}_{\leq 0}) \\
\text{Fun}(\mathcal{D},\mathcal{C})_{\geq 0}  \defeq & \text{Fun}(\mathcal{D},\mathcal{C}_{\geq 0}).
\end{eqnarray*}
We view the category of functors $\mathcal{D} \rightarrow \mathcal{C}_{\leq 0}$ as the full subcategory of functors $\mathcal{D} \rightarrow \mathcal{C}$ with values in the subcategory $\mathcal{C}_{\leq 0}$ and similarly for $\geq 0$.

\end{definition}

\begin{proof}
We have to check three things:
\begin{itemize}
\item
$\text{Fun}(\mathcal{D},\mathcal{C})_{\leq 0}$ is closed under $\Omega$. This is true since limits are computed object-wise and $\mathcal C_{\leq 0}$ is closed under limits. Similarly, $\text{Fun}(\mathcal{D},\mathcal{C})_{\geq 0}$ is closed under $\Sigma$.
\item Given $X$ in $\text{Fun}(\mathcal{D},\mathcal{C})_{\geq 1}$ and $Y$ in $\text{Fun}(\mathcal{D},\mathcal{C})_{\leq 0}$ the abelian group $\pi_0 \text{Nat}(X,Y)$ is zero. To see this, note that the valuewise adjunction between the inclusion of $\mathcal{C}_{\geq 1}$ in $\mathcal{C}$ and $\tau_{\geq 1}$ induces an adjunction on the functor categories. This gives the equivalence of mapping spaces (but not mapping spectra!)
$$\text{Nat}_\mathcal{C}(X,Y) \simeq \text{Nat}_\mathcal{C}(X,\tau_{\geq 1} Y) \simeq 0$$
since $\tau_{\geq 1} Y$ is value-wise the zero object, and therefore zero in $\text{Fun}(\mathcal{D},\mathcal{C})$.
\item For any $X$ in $\text{Fun}(\mathcal{D},\mathcal{C})$ there is a fiber sequence
$$X_1 \rightarrow X \rightarrow X_0$$
with $X_1$ in $\text{Fun}(\mathcal{D},\mathcal{C})_{\geq 1}$ and $X_0$ in $\text{Fun}(\mathcal{D},\mathcal{C})_{\leq 0}$. To see this, note that we have a fiber sequence $\tau_{\geq 1} \rightarrow \text{id}_{\mathcal C} \rightarrow \tau_{\leq 0}$ of functors $\mathcal{C} \rightarrow \mathcal{C}$.  Precomposing with $X$ gives the fiber sequence
$$\tau_{\geq 1} X \rightarrow X \rightarrow \tau_{\leq 0}, X$$
which is our desired fiber sequence.
\end{itemize} \end{proof}

\begin{remark}
The heart of this $t$-structure is given as
$$ \text{Fun}(\mathcal{D}, \mathcal{C})^\heartsuit \simeq \text{Fun}(\mathcal{D}, \mathcal{C}^\heartsuit) \simeq \text{Fun}(h\mathcal{D}, \mathcal{C}^\heartsuit),$$
where the right equivalence follows from $\mathcal{C}^\heartsuit$ being a $1$-category.
\end{remark}

\begingroup
\setlength{\emergencystretch}{8em}
\printbibliography
\endgroup

\end{document}